\documentclass[10pt,reqno]{amsart}
\usepackage{amssymb, latexsym}
\usepackage{xcolor}
\usepackage{amscd}
\usepackage{amssymb}
\usepackage{latexsym}

\usepackage{amsmath}
\usepackage{amscd}
\usepackage{cite}
\usepackage{color}
\usepackage{enumerate}
\usepackage{amsfonts}
\usepackage{graphicx}
\usepackage{mathrsfs}
\usepackage{setspace}
\usepackage{hyperref}

\usepackage{pgfplots,pgfplotstable,placeins}
\usepgfplotslibrary{groupplots,fillbetween}
\pgfplotsset{compat=1.18}

\pdfmapfile{+cm.map}
\pdfmapfile{+cmextra.map}
\pdfmapfile{+symbols.map}
\pdfmapfile{+latxfont.map}

\def\bi{\begin{itemize}}
\def\bs{\begin{split}}
\def\es{\end{split}}
\def\ba{\begin{align}}
\def\bas{\begin{align*}}
\def\ea{\end{align}}
\def\eas{\end{align*}}
\def\Im{{\operatorname{Im}}}

\def\C{{\mathbb C}} % Jim changed to mathbb
\def\R{{{\mathbb R}}}
\def\Z{{{\mathbb Z}}}

\def\emph#1{{\it #1}}
\def\textbf#1{{\bf #1}}

\newcommand{\lnr}{\langle\nabla\rangle}

\newcommand{\hk}{{\dot H^k}}

\newcommand{\eps}{{\varepsilon}}

\newtheorem{theorem}{Theorem}[section]

\newtheorem{lemma}[theorem]{Lemma}

\newtheorem{proposition}[theorem]{Proposition}

\theoremstyle{definition}
\newtheorem{definition}[theorem]{Definition}

\numberwithin{equation}{section} \numberwithin{theorem}{section}

\begin{document}

\title[GWP and scattering for NLS]
{Global well-posedness and scattering for the defocusing nonlinear Schr\"odinger equation below energy space}
\author[Z. Ma]{Zuyu Ma}
	\address{Zuyu Ma
		\newline \indent Institute of Applied Physics and Computational Mathematics, Beijing, 100088, China}
	\email{mazuyu23@gscaep.ac.cn}
    
 \author[C. Miao]{Changxing Miao}
\address{Changxing Miao 
\newline \indent School of Mathematics and Physics, University of Science and Technology Beijing, Beijing 100083, China}
\email{miao\_changxing@ustb.edu.cn, miao\_changxing@iapcm.ac.cn}

\author[Y. Song]{Yilin Song}
	\address{Yilin Song
		\newline \indent Institute of Applied Physics and Computational Mathematics, Beijing, 100088, China}
	\email{songyilin21@gscaep.ac.cn}

	\author[J. Zheng]{Jiqiang Zheng}
	\address{Jiqiang Zheng
		\newline Institute of Applied Physics and Computational Mathematics, Beijing, 100088, China.
		\newline National Key Laboratory of Computational Physics, Beijing, 100088, China}
	\email{zheng\_jiqiang@iapcm.ac.cn, zhengjiqiang@gmail.com}	
	
\subjclass[2020]{35Q55} \keywords{Nonlinear Schr\"odinger equation, well-posedness, scattering, low regularity}

\begin{abstract}
In this paper, we prove global well-posedness and scattering below the energy
space for the defocusing nonlinear Schr\"odinger equation with intercritical power
nonlinearities in spatial dimensions $d\geq3$, improving the regularity threshold
of Vi\c{s}an and Zhang[Differ. Integr. Equations,  22(2009), 99-124.].
\end{abstract}

\maketitle

\section{Introduction}

We consider the initial-value problem for the defocusing nonlinear
Schr\"odinger equation
\begin{equation}\label{equation}\tag{NLS}
\begin{cases}
i u_t+\Delta u=|u|^p u,\\
u(0,x)=u_0(x)\in H^s(\R^d),
\end{cases}
\end{equation}
where $d\ge3$ and $\frac4d<p<\frac4{d-2}$.
The equation \eqref{equation} is invariant under the scaling tranformation
\begin{equation}\label{scaling}
u(t,x)\mapsto u_\lambda(t,x)
:=\lambda^{-\frac2p}u\Bigl(\frac{t}{\lambda^2},\frac{x}{\lambda}\Bigr),
\qquad \lambda>0.
\end{equation}
It leaves the $\dot H^{s_c}(\R^d)$-norm invariant, where
$s_c:=\frac d2-\frac2p\in(0,1)$ is the scaling-critical regularity.
Thus \eqref{equation} is mass-supercritical and energy-subcritical in our setting.
For solutions in $H^1(\R^d)$, the  mass and energy are conserved:
\begin{align}
M(u(t))&=\int_{\R^d}|u(t,x)|^2\,dx,\label{mass}\\
E(u(t))&=\int_{\R^d}
\left[\frac12|\nabla u(t,x)|^2
+\frac1{p+2}|u(t,x)|^{p+2}\right]dx.\label{energy}
\end{align}

\subsection{Prior work}\label{subsec:prior-work}

The local well-posedness theory for \eqref{equation} in $H^s(\R^d)$ with
$s_c<s\le1$ was established by Cazenave-Weissler \cite{cwI} and we refer to \cite{cazenave:book} for more details.
In this intercritical range, the lifespan of the local solution depends only on the
$H^s$-norm of the initial data. At the energy level $s=1$, by the local well-posedness and conservation of mass and
energy gives the global
well-posedness in $H^1(\R^d)$.
Ginibre and Velo \cite{gv:scatter} further proved the $H^1$ scattering in  the sense that for every $u_0\in H^1(\R^d)$, there exist
$u_\pm\in H^1(\R^d)$ such that
\begin{equation}\label{intro energy scattering}
\lim_{t\to\pm\infty}
\|u(t)-e^{it\Delta}u_\pm\|_{H^1(\R^d)}=0.
\end{equation}
Here $e^{it\Delta}$ denotes the free Schr\"odinger propagator. This property implies that the
nonlinear solution is asymptotically closed to a linear scattering state in $H^1$ for both time directions.

For $s_c<s<1$, the lack of conservation laws at the inter-critical regularity makes it difficult to obtain the global well-posedness below $H^1$.
For the three-dimensional cubic equation, Bourgain \cite{Blowscatter}
developed the high--low frequency decomposition to prove global
well-posedness in $H^s(\R^3)$ for $s>\frac{11}{13}$.
Colliander, Keel, Staffilani, Takaoka, and Tao subsequently introduced
the $I$-method \cite{ckstt:low1} and show the global well-posedness by the almost conservation law. The key contribution is that the $I$-operator can map $H^s\to H^1$. Therefore, they only need to control the energy increment $E(Iu)-E(Iu_0)$ since the modified equation $Iu$ is not a Hamilton system.
Combining this method with the interaction
Morawetz inequality, they proved global well-posedness and scattering
for $s>\frac45$ in \cite{ckstt:low}.

For the same equation, Dodson \cite{Dodson-DCDS} obtained global
well-posedness and scattering for $s>\frac57$ using a
linear--nonlinear decomposition.
Su \cite{Su} improved this to $s>\frac{49}{74}$ by combining
linear--nonlinear decomposition, resonance decomposition \cite{CKSTT-DCDS}, and
large time iteration.
For radial initial data, Dodson \cite{Dodson-Camb} proved global
well-posedness and scattering for every $s>\frac12$ using
long-time Strichartz estimates. Recently, Su and Zhao \cite{SZ} remove the radial assumption of \cite{Dodson-Camb} by combining the I-method with improved long-time bilinear $L^2_{t,x}$ estimates for frequency-localized components of the solution.

For the non-algebraic nonlinearity, it is very difficult to use $I$-method directly. Vi\c{s}an and Zhang \cite{VZ07} treated general power-type
nonlinearities by
 combining the $I$-method with the interaction Morawetz inequality. The main ingredient is the nonlinear commutator estimate which was proved using the Bony paraproduct decomposition. 

\subsection{Main result}\label{subsec:main-result}

In this paper, our goal is to obtain  a lower regularity threshold for \eqref{equation} compared to \cite{VZ07}.
To state our result, we introduce several notations.
\begin{equation}\label{new exponents}
\nu:=\min\{1,p\},\qquad
\delta:=\nu(1-s_c),\qquad
A_d(\sigma):=d-3-\sigma(d-6).
\end{equation}
Let $\rho_1(d,p)$ be the unique root in $(s_c,1)$ of
\begin{equation}\label{rho definition}
\delta\rho(\rho-s_c)=s_c(1-\rho)A_d(\rho),
\end{equation}
and set
\begin{equation}\label{optimized threshold}
\rho_2(d,p):=\frac{d-2+\nu s_c}{d-2+\nu},\qquad
s_*(d,p):=\max\{\rho_1(d,p),\rho_2(d,p)\}.
\end{equation}
Then $\max\{s_c,\frac12\}<s_*(d,p)<1$.\footnote{On
$(s_c,1)$, the function $\delta(\rho-s_c)/(1-\rho)$ increases strictly
from zero to infinity, whereas
$s_cA_d(\rho)/\rho=s_c\{(d-3)/\rho-(d-6)\}$ is positive and
nonincreasing. This proves the existence and uniqueness of $\rho_1$;
the remaining inequalities follow from the formula for $\rho_2$.}

\begin{theorem}\label{scattering}
Let $d\ge3$, $\frac4d<p<\frac4{d-2}$, and $u_0\in H^s(\R^d)$, where
\begin{equation}\label{main threshold}
 s_*(d,p)<s<1.
\end{equation}
Then \eqref{equation} is globally well-posed in $H^s(\R^d)$, and
its global solution $u$ satisfies
\[
\|u\|_{L_t^\infty H_x^s(\R\times\R^d)}
\le C\bigl(\|u_0\|_{H^s(\R^d)}\bigr).
\]
Moreover, there exist unique scattering states $u_\pm\in H^s(\R^d)$
such that
\[
\lim_{t\to\pm\infty}
\|u(t)-e^{it\Delta}u_\pm\|_{H^s(\R^d)}=0.
\]
\end{theorem}

We compare this result with the regularity threshold of
Vi\c{s}an and Zhang \cite{VZ07}. In our notation, let
\[
s_1(d,p):=\frac{dp}{2(p+2)},\qquad
s_2(d,p):=\frac{1+\nu s_c}{1+\nu},
\]
and, for $0<\sigma<1$, let $s_+(d,p,\sigma)$ denote the root in
$(s_c,1)$ of
\[
s_c(1-s)A_d(\sigma)=\nu\sigma(s-s_c)^2.
\]
For $d\ge4$, taking the infimum over the admissible choices of
$\sigma$ gives
\begin{equation}\label{intro VZ optimized}
\begin{split}
s_{\mathrm{VZ}}(d,p):=\inf\bigl\{s\in(s_c,1):\;&
\text{there exists }0<\sigma\le s\text{ such that}\\
&s>\max\{s_1(d,p),s_2(d,p),s_+(d,p,\sigma)\},\\
&2\sigma[8-p(d+2)]<(d-3)(dp-4)\bigr\}.
\end{split}
\end{equation}
In dimension three, the explicit threshold in
\cite[Remarks following (1.4)]{VZ07} is
\begin{equation}\label{intro VZ three}
s_{\mathrm{VZ}}(3,p)
=\max\left\{\frac{3p}{2(p+2)},
\frac{-s_c+\sqrt{12s_c-3s_c^2}}2\right\}.
\end{equation}
After some tedious but straightforward calculation, one can verify that
these thresholds satisfy
\[
s_*(d,p)<s_{\mathrm{VZ}}(d,p),
\qquad d\ge3,\qquad \frac4d<p<\frac4{d-2}.
\]
Thus Theorem~\ref{scattering} improved the result of
\cite{VZ07} to the range $s_*(d,p)<s\le s_{\mathrm{VZ}}(d,p)$.
Figure~\ref{fig:threshold-comparison} compares the thresholds
in dimensions $3,4,5,6$.

\begingroup
\pgfplotstableread{
p new old
1.333333333 0.5 0.6
1.357777778 0.513502455 0.606551952
1.382222222 0.526527331 0.613009198
1.406666667 0.539099526 0.619373777
1.431111111 0.551242236 0.625647668
1.455555556 0.562977099 0.631832797
1.482222222 0.575337331 0.638481174
1.533333333 0.597826087 0.650943396
1.535555556 0.598769899 0.651798495
1.551111111 0.60530086 0.668353032
1.566666667 0.611702128 0.683772995
1.582222222 0.617977528 0.698177529
1.6 0.625 0.713525492
1.617777778 0.645110285 0.727806911
1.635555556 0.663860596 0.741130444
1.653333333 0.681369322 0.753589143
1.671111111 0.697742004 0.765263382
1.691111111 0.714920449 0.777545991
1.711111111 0.730899959 0.789010733
1.731111111 0.745786772 0.799732877
1.753333333 0.761160789 0.810855203
1.775555556 0.775417472 0.82122074
1.797777778 0.788658709 0.830898404
1.822222222 0.802158866 0.840821368
1.846666667 0.814646167 0.850055852
1.864444444 0.82314586 0.856375145
1.882222222 0.831192436 0.862384653
1.9 0.838815789 0.868103949
1.92 0.84692029 0.874213556
1.94 0.854559662 0.880001926
1.96 0.861766582 0.885490384
1.98 0.868570913 0.890698428
2.002222222 0.875692289 0.896177943
2.046666667 0.888678966 0.906252125
2.093333333 0.900725299 0.915700078
2.142222222 0.91184389 0.924518466
2.193333333 0.922058337 0.932711871
2.224444444 0.927658435 0.937244881
2.255555556 0.932840904 0.94146768
2.288888889 0.93797079 0.945675623
2.322222222 0.942701842 0.949582624
2.393333333 0.951619164 0.957020883
2.468888889 0.959597993 0.963767295
2.548888889 0.966664254 0.969823554
2.635555556 0.973006807 0.975333884
2.728888889 0.978596041 0.980256682
2.828888889 0.983431888 0.984574567
2.935555556 0.987536785 0.988289785
3.048888889 0.990949293 0.991419781
3.173333333 0.993807032 0.994075889
3.308888889 0.996092576 0.99622859
3.457777778 0.997837805 0.997894313
3.62 0.999048437 0.999065085
3.8 0.999765038 0.999767099
4 1 1
}\VZplotDataA
\pgfplotstableread{
p new old
1 0.666666667 0.732050808
1.026666667 0.683982684 0.745970246
1.054166667 0.700922266 0.759587088
1.0825 0.717474981 0.772892948
1.110833333 0.733183296 0.785520039
1.140833333 0.748965181 0.798206269
1.170833333 0.763938316 0.810242387
1.201666667 0.778548313 0.821986598
1.234166667 0.793157776 0.83373038
1.235 0.793715662 0.834023375
1.238333333 0.797435323 0.835234979
1.25 0.809913999 0.844710139
1.2625 0.822407276 0.85423496
1.276666667 0.835558647 0.864309738
1.290833333 0.847729371 0.87368381
1.305833333 0.859640152 0.882911224
1.321666667 0.871219676 0.891939387
1.3375 0.881870472 0.900299539
1.354166667 0.892169178 0.908440452
1.3675 0.899791026 0.914505298
1.380833333 0.906907395 0.920201442
1.394166667 0.913554372 0.925553486
1.408333333 0.920139345 0.930888124
1.4225 0.926267844 0.935884332
1.4375 0.932295978 0.94083057
1.4525 0.937885161 0.945447087
1.4675 0.943068015 0.949756223
1.490833333 0.950391927 0.955896012
1.515 0.957122909 0.961596391
1.54 0.963272902 0.966858853
1.566666667 0.969025379 0.971833779
1.594166667 0.974186331 0.976346026
1.6225 0.978782096 0.980408324
1.6525 0.98294691 0.984130781
1.684166667 0.986658249 0.98748587
1.716666667 0.989829679 0.990385892
1.750833333 0.992558962 0.992910389
1.786666667 0.994845146 0.995049675
1.825 0.996729149 0.996833486
1.865 0.998165955 0.998210157
1.9075 0.999190178 0.999203261
1.9525 0.999799481 0.999801106
2 1 1
}\VZplotDataB
\pgfplotstableread{
p new old
0.8 0.789473684 0.820482332
0.845333333 0.80963939 0.838518471
0.892 0.829907503 0.856444578
0.939555556 0.85006769 0.874083128
0.988 0.870110331 0.891436896
0.990666667 0.872864068 0.892504016
1.008888889 0.892284209 0.908046157
1.020888889 0.9040579 0.917536274
1.035555556 0.916956419 0.928025936
1.051111111 0.929035709 0.937952391
1.067555556 0.9402192 0.947247399
1.084888889 0.950451342 0.955855785
1.103111111 0.959696015 0.963734637
1.122222222 0.967934739 0.970852462
1.143555556 0.975600811 0.977573492
1.166222222 0.982221267 0.983470451
1.190222222 0.98776906 0.988494273
1.215555556 0.992241788 0.992614254
1.242222222 0.995658133 0.99581653
1.270666667 0.998083562 0.998130707
1.300888889 0.999521708 0.999527668
1.333333333 1 1
}\VZplotDataC
\pgfplotstableread{
p new old
0.666666667 0.857142857 0.872983346
0.740833333 0.890666198 0.903649178
0.816944444 0.923995156 0.933589214
0.818055556 0.925203176 0.934117536
0.825833333 0.933417705 0.941024156
0.834722222 0.941930742 0.948235424
0.843888889 0.949822116 0.954975372
0.853333333 0.957098262 0.961243626
0.863055556 0.963766613 0.967040143
0.873333333 0.969993325 0.972504337
0.883611111 0.975451923 0.977341943
0.894444444 0.980450133 0.981816982
0.906111111 0.985045575 0.985976752
0.918333333 0.989066629 0.989659513
0.930833333 0.992419617 0.992768138
0.943611111 0.995128974 0.995311728
0.956944444 0.997255785 0.997334462
0.970833333 0.998783621 0.99880726
0.985277778 0.999700772 0.999703708
1 1 1
}\VZplotDataD

\newcommand{\VZpanel}[5]{%
  \nextgroupplot[title={$d=#1$},xmin=4/#1,xmax=4/(#1-2),ymin=#3,
    legend to name=vzcomparisonlegend#1]
  \addplot[black,thick,name path=VZnew] table[x=p,y=new] {#2};
  \addplot[black,thick,dashed,name path=VZold] table[x=p,y=old] {#2};
  \addplot[fill=black!15,draw=none,forget plot]
    fill between[of=VZnew and VZold];
  \addplot[black,only marks,mark=o,mark size=1.6pt,
    mark options={fill=white},forget plot]
    coordinates {({4/#1},#4) ({4/#1},#5) ({4/(#1-2)},1)};
  \ifnum#1=3
    \addlegendentry{Present threshold $s_*(d,p)$}
    \addlegendentry{Vi\c{s}an--Zhang $s_{\mathrm{VZ}}(d,p)$}
    \addlegendimage{area legend,draw=none,fill=black!15}
    \addlegendentry{Additional regularity range}
  \fi
}
\begin{figure}[htbp]
\centering
\resizebox{\textwidth}{!}{%
\begin{tikzpicture}
\begin{groupplot}[
  group style={group size=2 by 2,horizontal sep=1.6cm,vertical sep=1.8cm},
  width=7.4cm,height=4.0cm,
  xlabel={Nonlinearity exponent $p$},ylabel={Regularity threshold $s$},
  axis lines=left,axis line style={-},
  ymajorgrids,grid style={black!15},tick align=outside,scaled ticks=false,
  tick label style={font=\small,/pgf/number format/fixed},
  label style={font=\small},title style={font=\large},
  ymax=1.008,clip=false,legend columns=3,
  legend style={draw=none,font=\small,cells={anchor=west}}]
\VZpanel{3}{\VZplotDataA}{0.45}{0.5}{0.6}
\VZpanel{4}{\VZplotDataB}{0.6}{0.666666667}{0.732050808}
\VZpanel{5}{\VZplotDataC}{0.75}{0.789473684}{0.820482332}
\VZpanel{6}{\VZplotDataD}{0.8}{0.857142857}{0.872983346}
\end{groupplot}
\node[anchor=south,yshift=3mm] at (current bounding box.north)
  {\pgfplotslegendfromname{vzcomparisonlegend3}};
\end{tikzpicture}%
}
\caption{Regularity thresholds for $d=3,4,5,6$.}
\label{fig:threshold-comparison}
\end{figure}
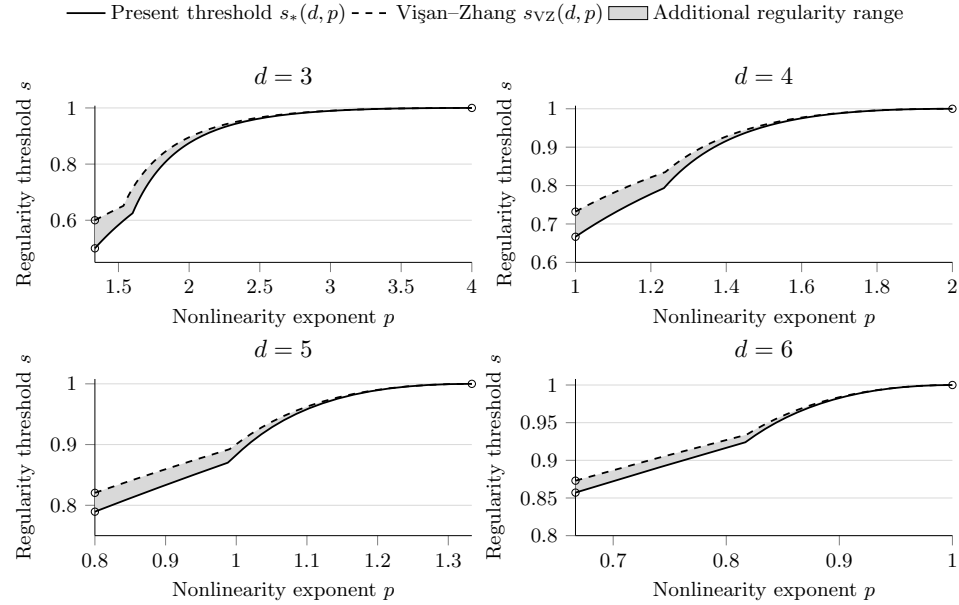
\endgroup
\FloatBarrier

\subsection{Idea of the proof}\label{subsec:proof-idea}

The proof combines the $I$-method with the interaction Morawetz
inequality, following \cite{ckstt:low,VZ07}.

The classical $I$-method introduced by \cite{ckstt:low} developed a smoothing operator $I_N:H^s\to H^1$
and and reduces global well-posedness to estimating the increment of the modified energy 
$E(I_Nu)$, since the modified equation is no longer a Hamiltonian system.  For the cubic nonlinearity or more generally the algebraic nonlinearities, the energy increment can be expanded into a finite sum of multilinear expressions, so that each
frequency interaction can be estimated separately. For the general power-type nonlinearity, however, the nonlinearity is non-algebraic, and the multilinear expansion used for algebraic nonlinearities does not apply directly.  To treat such nonlinearities within the $I$-method, Vişan and Zhang \cite{VZ07,VZ09} used Bony’s paraproduct calculus to estimate the
commutator between $I_N$ and $F$. The increment of this energy is
estimated by controlling the nonlinear commutator
$I_NF(u)-F(I_Nu)$, where $F(z)=|z|^pz$. The main goal of our paper is to design the new $I_N$-operator which capture the cancellation structure in physical space.   Our key novelty is the construction of a smoothing operator  $I_N$
  that retains the main properties of the usual $I$-multiplier while simultaneously yielding stronger estimates for the increment of the modified energy. For algebraic nonlinearities, the corresponding modified-energy estimates are well
understood. Thus, the main new ingredient of this paper is a modified-energy estimate that applies for every 
$0<s<1$
 in the presence of non-algebraic nonlinearity. More importantly, this new multiplier is specifically tailored to handle non-algebraic power-type nonlinearities more effectively, thereby providing a substantially improved version of the 
regularity threshold in this setting. 

Following our previous work \cite{MMSZheat}, we define $I_N$ as a convex combination of
heat operators:
\[
I_N=(1-\vartheta)\sum_{j=0}^{\infty}
\vartheta^j e^{\Delta/(2^jN)^2},
\qquad \vartheta=2^{-(1-s)}.
\]
This operator is a contraction on $L^r$, $1\le r\le\infty$, and
gains $1-s$ derivatives at high frequencies.
The heat kernel representation and the identity
\[
I_N=(1-\vartheta)e^{\Delta/N^2}+\vartheta I_{2N}
\]
give a decomposition of the commutator in which the linear terms
in the Taylor expansion cancel. The superposition of weight $\theta^j$ makes us to obtain the similar pointwise multiplier bound as the classical one introduced in \cite{ckstt:low}:
$$
m_N(\xi)\sim_s
\begin{cases}
1,&|\xi|\le N,\\
\left(\dfrac{N}{|\xi|}\right)^{1-s},&|\xi|\ge N.
\end{cases}
$$
Moreover, we will show that our new $I$-operator enjoy the similar boundedness as the classical $I$-operator. 

Now, we can turn to prove the global well-posedness and scattering. 
The key ingredient to the improvement is the refined energy increment (see Proposition \ref{energy increment scattering}). We will achieve this by showing the better commutator estimate, which is based on our new $I$-operator. For this nonlinear commutator estimate, we refer to Section \ref{preli}.

\subsection*{Organization of the paper}\label{subsec:organization}

The rest of the paper is organized as follows.
Section~2 introduces the notation, collects preliminary estimates,
and proves the nonlinear commutator estimates.
In Section~3 we prove Theorem~\ref{scattering} and the almost
conservation law used in its proof.

\medskip

\subsection*{Acknowledgements}

C. Miao was supported by the National Key R\&D Program of China under
Grant 2021YFA1002500 and by NSFC Grants 12026407 and 12531005.
J. Zheng was supported by the National Key R\&D Program of China under
Grant 2021YFA1002500 and by NSFC Grant 12671284.

\section{Preliminaries}\label{preli}
We will often use the notation $X \lesssim Y$ whenever there exists some constant $C$ so that $X \leq CY$. Similarly, we will use $X
\sim Y$ if $X \lesssim Y \lesssim X$.  We use $X \ll Y$ if $X \leq cY$ for some small constant $c$. The derivative operator $\nabla$
refers to the space variable only.  We use $A\pm$ to denote $A\pm\eps$ for any sufficiently small $\eps>0$.

Let $F(z):=|z|^pz$ with $p\in(\frac4d,\frac{4}{d-2})$ with $d\geq3$.  Then,
$$
F_z(z):=\frac{\partial F}{\partial z}(z)=\frac{p+2}2|z|^{p} \quad \text{and}
\quad F_{\bar z}(z):=\frac{\partial F}{\partial \bar z}(z)=\frac{p}2|z|^{p}\frac{z}{\bar z}.
$$
We write $F'$ for the vector $(F_z,F_{\bar z})$ and adopt the notation
$$
w\cdot F'(z):=wF_{z}(z)+\bar wF_{\bar z}(z).
$$
In particular, we observe the chain rule
$$
\nabla F(u)=\nabla u\cdot F'(u).
$$
Clearly $F'(z)=O(|z|^{p})$ and we have the H\"older continuity estimate
\begin{align}\label{holder continuity}
|F'(z)-F'(w)|\lesssim |z-w|^{\min\{1,p\}}(|z|+|w|)^{p-\min\{1,p\}}
\end{align}
for all $z, w\in \C$.  By the mean-value theorem,
$$
F(z+w)-F(z)=\int_0^1 w\cdot F'(z+\theta w)d\theta
$$
and hence
$$
F(z+w)=F(z)+O(|w||z|^p)+O(|w|^{p+1})
$$
for all complex values $z$ and $w$.

The Lebesgue norm  $L_x^r(\R^d)$ can be defined via
$$
\|f\|_{L^r(\R^d)}:=\Bigl(\int_{\R^d} |f(x)|^r dx\Bigr)^{\frac{1}{r}}
$$
 with the usual modifications when $r=\infty$.  For any
 integer $k\geq0$, we denote by $H^{k,r}(\R^d)$ the Sobolev
space defined as the closure of test functions in the norm
$$
\|f\|_{H^{k,r}}:=\sum_{|\alpha|\leq k}\Bigl\|\frac{\partial
^\alpha}{\partial x^\alpha}f\Bigr\|_r.
$$
We will often denote $H^{k,2}$ by $H^k$.

For each time interval $I\subseteq\R$, we use $L_t^qL_x^r(I\times\R^d)$
to denote the mixed space-time norm
$$
\|u\|_{L_t^qL_x^r(I\times\R^d)}
:=\Bigl(\int_I\Bigl(\int_{\R^d}|u(t,x)|^r\,dx\Bigr)^{q/r}\,dt\Bigr)^{1/q},
$$
with the usual modifications when $q=\infty$ or $r=\infty$.

The Fourier transform on $\R^d$  is given by
$$
\hat f(\xi) := \int_{\R^d} e^{-i x \cdot \xi} f(x) dx.
$$

By using the functional calculus, the fractional differential operators $|\nabla|^s$ can be defined by
$$
\widehat{|\nabla|^sf}(\xi) := |\xi|^s \hat f (\xi).
$$
This defines the homogeneous Sobolev norms
$$
\|f\|_{\dot H^{s}_x} := \| |\nabla|^s f \|_{L^2_x}.
$$

Let $e^{it\Delta}$ be the free Schr\"odinger propagator.  The solution to this propagator can be represented by
$$
e^{it\Delta}f(x) = \frac{1}{(4 \pi i t)^{d/2}} \int_{\R^d} e^{i|x-y|^2/4t} f(y) dy
$$
for $t\neq 0$. 
In particular, the propagator obeys the dispersive estimate
\begin{equation}\label{dispersive ineq}
\|e^{it\Delta}f\|_{L^\infty_x} \lesssim
|t|^{-\frac{d}{2}}\|f\|_{L^1_x}
\end{equation}
for all times $t\neq 0$.

We also recall the Duhamel formula
\begin{align}\label{duhamel}
u(t) = e^{i(t-t_0)\Delta}u(t_0) - i \int_{t_0}^t e^{i(t-s)\Delta}(iu_t + \Delta u)(s) ds.
\end{align}

\begin{definition}
A pair of exponents $(q,r)$ is called Schr\"odinger-\emph{admissible} if
$$
\frac{2}{q} +\frac{d}{r} = \frac{d}{2},\quad 2 \leq q,r \leq \infty, \quad \text{and} \quad (q,r,d)\neq (2,\infty,2).
$$
\end{definition}
We will use the following admissible pairs:
\[
(2,\tfrac{2d}{d-2}) \quad\text{and}\quad
(\tfrac{d-3+4\sigma}{\sigma},
 \tfrac{2d(d-3+4\sigma)}{d(d-3+4\sigma)-4\sigma}),\qquad \sigma>0,
\]
and
\[
(\tfrac{2p(2+\varepsilon)}{\varepsilon},
 \tfrac{2dp(2+\varepsilon)}{dp(2+\varepsilon)-2\varepsilon})
\quad\text{and}\quad
(2+\varepsilon,\tfrac{2d(2+\varepsilon)}{d(2+\varepsilon)-4}),
\]
where $\varepsilon>0$ is chosen so that
$2p(2+\varepsilon)/\varepsilon\ge2$.
The choice used in the almost conservation law is specified in
Lemma~\ref{interpolation1}. 
On each time interval $I\subseteq\R$, we define the following norms, which will be frequently used in the proof of the main theorem:
\begin{align}
 \|u\|_{M_\sigma(I\times\R^d)}
 &:=\|u\|_{L_t^{\frac{d-3+4\sigma}{\sigma}}L_x^{\frac{2(d-3+4\sigma)}{d-3+2\sigma}}(I\times\R^d)},\label{Morawetz norm}\\
 Z_I(I)&:=\sup_{(q,r)\ {\rm admissible}}\|\nabla I_Nu\|_{L_t^{q}L_x^{r}(I\times\R^d)},\label{ZI}\\
 \|u\|_{W(I)}&:=\sup_{(q,r)\ {\rm admissible}}
                         \|\lnr^su\|_{L_t^{q}L_x^{r}(I\times\R^d)}.\label{W}
\end{align}
We record the standard Strichartz estimates used below. We refer to  \cite{keel-tao} for a detailed proof.

\begin{lemma}[Strichartz estimate, \cite{keel-tao}]\label{lemma linear strichartz}
Let $I$ be a compact time interval with $t_0\in I$, $k$ be an arbitrary integer. Suppose that  $u$ is a solution to the forced Schr\"odinger equation
\begin{equation*}
i u_t + \Delta u =\sum_{i=1}^m F_i
\end{equation*}
for some functions $F_1$, $\cdots$, $F_m$.  Then,
\begin{equation}
\||\nabla| ^k u\|_{L_t^qL_x^r(I\times\R^d)} \lesssim \|u(t_0)\|_{\hk(\R^d)} +  \sum_{i=1}^m \||\nabla| ^k F_i\|_{L_t^{q_i'}L_x^{r_i'} (I\times\R^d)}
\end{equation}
for any admissible pairs $(q,r)$ and $(q_i,r_i)$, $1\leq i\leq m$.
\end{lemma}
We will also need the Littlewood-Paley theory.  Let $\varphi(\xi)$ be a smooth cut-off function supported in the ball
$|\xi| \leq 2$ and equal to one on the ball $|\xi| \leq 1$.  For each dyadic number $N \in 2^\Z$ we define the
Littlewood-Paley projections by
\begin{align*}
\widehat{P_{\leq N}f}(\xi) &:=  \varphi(\xi/N)\hat f (\xi),\\
\widehat{P_{> N}f}(\xi) &:=  [1-\varphi(\xi/N)]\hat f (\xi),\\
\widehat{P_N f}(\xi) &:=  [\varphi(\xi/N) - \varphi (2 \xi /N)] \hat f (\xi).
\end{align*}
Similarly we can define $P_{<N}$, $P_{\geq N}$, and $P_{M < \cdot \leq N} := P_{\leq N} - P_{\leq M}$, whenever $M$ and
$N$ are dyadic numbers.  We will frequently write $f_{\leq N}$ for $P_{\leq N} f$ and similarly for the other operators.
We recall the following standard Bernstein and Sobolev type inequalities:

\begin{lemma}[Bernstein estimate, \cite{KV-book}]\label{bernstein}
For any $1\le p\le q\le\infty$ and $s>0$, we have
\begin{align*}
\|P_{\geq N} f\|_{L^p_x} &\lesssim N^{-s} \| |\nabla|^s P_{\geq N} f \|_{L^p_x}\\
\| |\nabla|^s  P_{\leq N} f\|_{L^p_x} &\lesssim N^{s} \| P_{\leq N} f\|_{L^p_x}\\
\| |\nabla|^{\pm s} P_N f\|_{L^p_x} &\sim N^{\pm s} \| P_N f \|_{L^p_x}\\
\|P_{\leq N} f\|_{L^q_x} &\lesssim N^{\frac{d}{p}-\frac{d}{q}} \|P_{\leq N} f\|_{L^p_x}\\
\|P_N f\|_{L^q_x} &\lesssim N^{\frac{d}{p}-\frac{d}{q}} \| P_N f\|_{L^p_x}.
\end{align*}
\end{lemma}

\vspace{0.4cm}

\subsection{The $I$-operator and nonlinear commutator estimates}
Following \cite{MMSZheat}, we define $I_N$ by averaging the heat semigroup.
Fix $0<s<1$ and set $\vartheta=2^{-(1-s)}$.
For $M>0$, write $S_M=e^{\Delta/M^2}$ and define the radial symbol
\[
m(r)=(1-\vartheta)\sum_{j\ge 0}\vartheta^j e^{-4^{-j}r^2}.
\]
We use the following positive average of heat operators:
\begin{equation}\label{semigroup I definition}
I_N=(1-\vartheta)\sum_{j=0}^{\infty}\vartheta^jS_{2^jN}.
\end{equation}
Thus $I_N$ is the Fourier multiplier with symbol
\begin{equation}\label{equ:mheat}
m_N(\xi)=m\big(\frac{|\xi|}{N}\big).
\end{equation}
The advantage of \eqref{semigroup I definition} is the exact recursion
\begin{equation}\label{I recursion}
I_N=(1-\vartheta)S_N+\vartheta I_{2N},
\end{equation}
which will be used to exploit the cancellation of the linear terms in
the Taylor expansion. We recall the properties of $I_N$ from
\cite[Proposition~2.3]{MMSZheat} and record the homogeneous estimates
that follow from the multiplier bounds in its proof.
Write $D_N:=|\nabla|I_N$. By boundedness of the Riesz transforms,
$\|D_Nf\|_{L^r}\sim_r\|\nabla I_Nf\|_{L^r}$ for $1<r<\infty$.
\begin{proposition}[Properties of $I_N$]\label{basic property}
Let $0<s<1$, $1<r<\infty$, and $N\ge1$. The implicit constants below
may depend on $d,s,r$ and on the indicated differentiation orders,
but not on $M$ or $N$.

\noindent $(i)$ The operator $I_N$ is a contraction on $L^r$:
\begin{equation}\label{i1}
\|I_Nf\|_{L^r}\le\|f\|_{L^r}.
\end{equation}
This also holds for $r=1$ and $r=\infty$.

\noindent $(ii)$ For every multi-index $\alpha$,
\begin{equation}\label{equ:msymbol}
|\partial_\xi^\alpha m_N(\xi)|
\lesssim_{s,\alpha}(N+|\xi|)^{-|\alpha|}m_N(\xi),
\end{equation}
and
\begin{equation}\label{symbol comparison}
m_N(\xi)\sim_s
\begin{cases}
1,&|\xi|\le N,\\
\left(\dfrac{N}{|\xi|}\right)^{1-s},&|\xi|\ge N.
\end{cases}
\end{equation}
Consequently, for $0\le\gamma\le s$,
\begin{equation}\label{i2}
\||\nabla|^\gamma P_{>N}f\|_{L^r}
\lesssim N^{\gamma-1}\|D_Nf\|_{L^r},
\end{equation}
while
\begin{equation}\label{i3}
\|f\|_{H^s}\lesssim_s\|\langle\nabla\rangle I_Nf\|_{L^2}
\lesssim_s N^{1-s}\|f\|_{H^s}.
\end{equation}
At low frequencies we also have
$\|\nabla P_{\le N}f\|_{L^r}\lesssim\|D_Nf\|_{L^r}$.

\noindent $(iii)$ If $M\ge N$, then
\begin{equation}\label{scale comparison}
\begin{aligned}
\|\langle\nabla\rangle I_Mf\|_{L^r}
&\lesssim\left(\frac{M}{N}\right)^{1-s}
 \|\langle\nabla\rangle I_Nf\|_{L^r},\\
\|D_Mf\|_{L^r}
&\lesssim\left(\frac{M}{N}\right)^{1-s}\|D_Nf\|_{L^r}.
\end{aligned}
\end{equation}
For every real $k\ge0$ and $0\le\rho\le2$,
\begin{equation}\label{equ:profileapprox}
\begin{aligned}
\|(1-I_N)f\|_{\dot H^k}
&\lesssim_{s,\rho}N^{-\rho}\|f\|_{\dot H^{k+\rho}},\\
\|(1-I_N)f\|_{H^k}
&\lesssim_{s,\rho}N^{-\rho}\|f\|_{H^{k+\rho}}.
\end{aligned}
\end{equation}
Moreover, for every $M\ge1$,
\begin{align}
\|(1-I_M)f\|_{L^r}+\|(1-S_M)f\|_{L^r}
&\lesssim M^{-1}\|D_Mf\|_{L^r},\label{equ:IMSMdiff}\\
\|\nabla S_Mf\|_{L^r}+\|\nabla I_Mf\|_{L^r}
&\lesssim\|D_Mf\|_{L^r},\label{derivative multiplier}\\
\|(S_M-I_M)f\|_{L^r}
&\lesssim M^{-1}\|D_Mf\|_{L^r}.\label{difference multiplier}
\end{align}
The differentiated versions needed below are, for $0\le\gamma\le s$,
\begin{align}
\||\nabla|^\gamma(1-I_M)f\|_{L^r}
 +\||\nabla|^\gamma(1-S_M)f\|_{L^r}
 &\lesssim M^{\gamma-1}\|D_Mf\|_{L^r},\label{estimate-11}\\
\||\nabla|^\gamma\nabla S_Mf\|_{L^r}
 &\lesssim M^\gamma\|D_Mf\|_{L^r}.\label{estimate-22}
\end{align}

In particular,
\begin{equation}\label{equ:L2tailI}
\|(1-I_N)f\|_{L^2}
\lesssim_sN^{-1}\|\langle\nabla\rangle I_Nf\|_{L^2}.
\end{equation}

\noindent $(iv)$ Let $\dot I_N$ denote the multiplier with symbol
$N\partial_Nm_N$. Then
\begin{equation}\label{equ:Idot}
\|\dot I_Nf\|_{L^2}
\lesssim_sN^{-1}\|\langle\nabla\rangle I_Nf\|_{L^2}.
\end{equation}
\end{proposition}
\begin{proof}
    We only give the detailed proof of \eqref{estimate-11} and \eqref{estimate-22}. Others can be found in our companian paper \cite{MMSZheat}. First, we prove \eqref{estimate-11}.  By making use of the high-low decomposition $f=P_{\leq M}f+P_{>M}f$, one has the estimates:
    \begin{align*}
       &\quad \||\nabla|^\gamma(I-I_M)P_{\leq M}f\|_{L^r}+
        \||\nabla|^\gamma(I-S_M)P_{\leq M}f\|_{L^r}\\
        &\lesssim M^\gamma(\|(I-I_M)P_{\leq M}f\|_{L^r}+\|(I-S_M)P_{\leq M}f\|_{L^r}\\
        &\lesssim M^{\gamma-1}\|D_{M}f\|_{L^r}
    \end{align*}
    and
    \begin{align*}
   &\quad \||\nabla|^\gamma(I-I_M)P_{> M}f\|_{L^r}+
        \||\nabla|^\gamma(I-S_M)P_{> M}f\|_{L^r} \\&\lesssim\||\nabla|^\gamma P_{>M}f\|_{L^r}\lesssim M^{\gamma-1}\|D_Mf\|_{L^r},
    \end{align*}
    which finishes the proof of \eqref{estimate-11}.  Now, we turn to prove \eqref{estimate-22}. By using the semigroup property of $S_M$ and \eqref{scale comparison}, one has
    \begin{align*}
        \||\nabla|^\gamma \nabla S_Mf\|_{L^r}&=\||\nabla|^\gamma S_{\sqrt 2M}(\nabla S_{\sqrt 2M}f)\|_{L^r}\lesssim (\sqrt{2}M)^\gamma\|\nabla S_{\sqrt 2M}f\|_{L^r}\\
        &\lesssim M^\gamma \|D_{\sqrt 2M}f\|_{L^r}\lesssim M^\gamma\|D_Mf\|_{L^r}.
    \end{align*}
    Hence, we complete the proof of this proposition. 
\end{proof}
\begin{lemma}[Nonlinear commutator]\label{heat commutator}
Let $d\ge3$, $\frac4d<p<\frac4{d-2}$, and $s_c<s<1$.
Fix $\varepsilon>0$ such that
$q_\varepsilon:=2p(2+\varepsilon)/\varepsilon\ge2$.
On every time interval $I\subseteq\R$, we have
\begin{align}
 \|\nabla[I_NF(u)-F(I_Nu)]\|_{L_t^{2}L_x^{\frac{2d}{d+2}}(I\times\R^d)}
 &\lesssim N^{-\delta}
 \|u\|_{L_t^{\frac{2p(2+\varepsilon)}{\varepsilon}}L_x^{\frac{dp(2+\varepsilon)}{4+\varepsilon}}(I\times\R^d)}^{p-\nu}Z_I(I)^{1+\nu},
 \label{gradient commutator estimate}\\
 \|I_NF(u)-F(I_Nu)\|_{L_t^{2}L_x^{\frac{2d}{d+2}}(I\times\R^d)}
 &\lesssim N^{-(1+\delta)}
 \|u\|_{L_t^{\frac{2p(2+\varepsilon)}{\varepsilon}}L_x^{\frac{dp(2+\varepsilon)}{4+\varepsilon}}(I\times\R^d)}^{p-\nu}Z_I(I)^{1+\nu}.
 \label{zero commutator estimate}
\end{align}
All spacetime norms are taken on $I\times\R^d$, and
$\delta=\nu(1-s_c)$. The implicit constants may depend on
$d,p,s,\varepsilon$, but not on $N$ or $I$.
\end{lemma}
\begin{proof}
It suffices to take $u\in C_c^\infty(I\times\R^d)$.
Let $R(z,w)$ be the first-order Taylor remainder
\[
 R(z,w):=F(z+w)-F(z)-w\cdot F'(z)
       =\int_0^1w\cdot[F'(z+\tau w)-F'(z)]\,d\tau,
\]
\eqref{holder continuity} gives
\begin{equation}\label{remainder estimate}
 |R(z,w)|\lesssim (|z|+|w|)^{p-\nu}|w|^{1+\nu}.
\end{equation}
Let us denote the parameters in the following:
\[
 \vartheta=2^{-(1-s)},\qquad N_j=2^jN,\qquad
 c_j=(1-\vartheta)\vartheta^j,\qquad \sum_{j\ge0}c_j=1.
\]
By \eqref{semigroup I definition}, we have the following decomposition:
\begin{equation}\label{commutator AB decomposition}
\begin{aligned}
 \mathcal C_N(u)
 &:=I_NF(u)-F(I_Nu)\\
 &=\sum_{j\ge0}c_j\{S_{N_j}F(u)-F(S_{N_j}u)\}+\left\{\sum_{j\ge0}c_jF(S_{N_j}u)-F(I_Nu)\right\}\\
 &=:\mathcal A_N+\mathcal B_N.
\end{aligned}
\end{equation}

\medskip
\noindent\textbf{Step 1: The estimate of $\mathcal A_N$.}
For $M>0$, the heat operator $S_M$ can be represented by the kernel $K_M$,
\begin{equation}\label{heat convolution identity}
 S_Mf(x)=\int_{\R^d}K_M(y)f(x-y)\,dy.
\end{equation}
Thus
\begin{equation}\label{heat kernel cancellation}
 \int_{\R^d}K_M(y)\,dy=1,\qquad
 \int_{\R^d}\nabla K_M(y)\,dy=0.
\end{equation}
Set
\[
 H_M:=S_MF(u)-F(S_Mu),\qquad
 h_y(t,x):=u(t,x-y)-S_Mu(t,x).
\]
Since
\[
 \int_{\R^d}K_M(y)h_y(t,x)\,dy
 =S_Mu(t,x)-S_Mu(t,x)=0,
\]
we have
\begin{align*}
 &\int_{\R^d}K_M(y)R(S_Mu,h_y)\,dy\\
 &\quad=\int_{\R^d}K_M(y)F(u(t,x-y))\,dy-F(S_Mu)\\
 &\qquad-\left[\int_{\R^d}K_M(y)h_y\,dy\right]\cdot F'(S_Mu)\\
 &\quad=S_MF(u)-F(S_Mu).
\end{align*}
Consequently,
\begin{equation}\label{heat defect}
 H_M(t,x)=\int_{\R^d}K_M(y)R(S_Mu,h_y)\,dy.
\end{equation}
Integration by parts gives
\[
 \nabla S_Mf(x)=\int_{\R^d}\nabla K_M(y)f(x-y)\,dy.
\]
Therefore,
\begin{align*}
 \nabla H_M(t,x)
 & =\int_{\R^d}\nabla K_M(y)F(u(t,x-y))\,dy\\
 &\quad-\left[\int_{\R^d}\nabla K_M(y)u(t,x-y)\,dy\right]
              \cdot F'(S_Mu)\\
 &=\int_{\R^d}\nabla K_M(y)
       [F(u(t,x-y))-F(S_Mu)]\,dy\\
 &\quad-\left[\int_{\R^d}\nabla K_M(y)
       [u(t,x-y)-S_Mu]\,dy\right]\cdot F'(S_Mu).
\end{align*}
Hence
\begin{equation}\label{gradient heat defect}
 \nabla H_M(t,x)=\int_{\R^d}\nabla K_M(y)R(S_Mu,h_y)\,dy.
\end{equation}

The identity
\[
 h_y=(1-S_M)u(\cdot-y)
       -\int_0^1y\cdot\nabla S_Mu(\cdot-\tau y)\,d\tau
\]
and Proposition~\ref{basic property} imply, for $M\ge N$,
$0\le\gamma\le s$, $1\le q\le\infty$, and $1<r<\infty$,
\begin{equation}\label{general heat difference}
\begin{aligned}
 \||\nabla|^\gamma h_y\|_{L_t^qL_x^r(I\times\R^d)}
 &\le\||\nabla|^\gamma(1-S_M)u\|_{L_t^qL_x^r(I\times\R^d)}
       +|y|\,\||\nabla|^\gamma\nabla S_Mu\|_{L_t^qL_x^r(I\times\R^d)}\\
 &\lesssim N^{s-1}M^{\gamma-s}(1+M|y|)
              \|\nabla I_Nu\|_{L_t^qL_x^r(I\times\R^d)}.
\end{aligned}
\end{equation}
Taking $M=N_j$ in \eqref{general heat difference}, Sobolev embedding gives
\begin{align}
 \|h_y\|_{L_t^{2+\varepsilon}L_x^{\frac{2d(2+\varepsilon)}{d(2+\varepsilon)-4}}(I\times\R^d)}
 &\lesssim N^{-1}2^{-sj}(1+N_j|y|)Z_I(I),
 \label{difference ess}\\
 \|h_y\|_{L_t^{\frac{2p(2+\varepsilon)}{\varepsilon}}L_x^{\frac{dp(2+\varepsilon)}{4+\varepsilon}}(I\times\R^d)}
 &\lesssim\||\nabla|^{s_c}h_y\|_{L_t^{\frac{2p(2+\varepsilon)}{\varepsilon}}L_x^{\frac{2dp(2+\varepsilon)}{dp(2+\varepsilon)-2\varepsilon}}(I\times\R^d)}\notag\\
 &\lesssim N^{-(1-s_c)}2^{-(s-s_c)j}(1+N_j|y|)Z_I(I).
 \label{difference ep}
\end{align}
By the $L^p$-boundedness of the heat semigroup, it holds
\[
 \begin{aligned}
 &\|S_{N_j}u\|_{L_t^{\frac{2p(2+\varepsilon)}{\varepsilon}}L_x^{\frac{dp(2+\varepsilon)}{4+\varepsilon}}(I\times\R^d)}+\|h_y\|_{L_t^{\frac{2p(2+\varepsilon)}{\varepsilon}}L_x^{\frac{dp(2+\varepsilon)}{4+\varepsilon}}(I\times\R^d)}\\
 &\quad\le3\|u\|_{L_t^{\frac{2p(2+\varepsilon)}{\varepsilon}}L_x^{\frac{dp(2+\varepsilon)}{4+\varepsilon}}(I\times\R^d)}.
 \end{aligned}
\]
The H\"older inequality yield
\begin{equation}\label{heat remainder spacetime}
\begin{aligned}
 &\|R(S_{N_j}u,h_y)\|_{L_t^2L_x^{\frac{2d}{d+2}}(I\times\R^d)}\\
 &\quad\lesssim\|u\|_{L_t^{\frac{2p(2+\varepsilon)}{\varepsilon}}L_x^{\frac{dp(2+\varepsilon)}{4+\varepsilon}}(I\times\R^d)}^{p-\nu}
              \|h_y\|_{L_t^{\frac{2p(2+\varepsilon)}{\varepsilon}}L_x^{\frac{dp(2+\varepsilon)}{4+\varepsilon}}(I\times\R^d)}^{\nu}\|h_y\|_{L_t^{2+\varepsilon}L_x^{\frac{2d(2+\varepsilon)}{d(2+\varepsilon)-4}}(I\times\R^d)}\\
 &\quad\lesssim N^{-(1+\delta)}2^{-\{s+\nu(s-s_c)\}j}
                     (1+N_j|y|)^{1+\nu}\\
 &\qquad\times\|u\|_{L_t^{\frac{2p(2+\varepsilon)}{\varepsilon}}L_x^{\frac{dp(2+\varepsilon)}{4+\varepsilon}}(I\times\R^d)}^{p-\nu}Z_I(I)^{1+\nu}.
\end{aligned}
\end{equation}

Using the dilation property of the heat kernel: $K_M(y)=M^dK_1(My)$, we obtain
\begin{align*}
 \int_{\R^d}K_M(y)(1+M|y|)^{1+\nu}\,dy
 &=\int_{\R^d}K_1(z)(1+|z|)^{1+\nu}\,dz\lesssim1,\\
 \int_{\R^d}|\nabla K_M(y)|(1+M|y|)^{1+\nu}\,dy
 &=M\int_{\R^d}|\nabla K_1(z)|(1+|z|)^{1+\nu}\,dz
 \lesssim M.
\end{align*}
Minkowski's inequality and \eqref{heat defect}--\eqref{gradient heat defect}
therefore give
\begin{align}
 \|H_{N_j}\|_{L_t^2L_x^{\frac{2d}{d+2}}(I\times\R^d)}
 &\lesssim N^{-(1+\delta)}2^{-\{s+\nu(s-s_c)\}j}
       \|u\|_{L_t^{\frac{2p(2+\varepsilon)}{\varepsilon}}L_x^{\frac{dp(2+\varepsilon)}{4+\varepsilon}}(I\times\R^d)}^{p-\nu}Z_I(I)^{1+\nu},\label{Hj zero}\\
 \|\nabla H_{N_j}\|_{L_t^2L_x^{\frac{2d}{d+2}}(I\times\R^d)}
 &\lesssim N^{-\delta}2^{\{1-s-\nu(s-s_c)\}j}
       \|u\|_{L_t^{\frac{2p(2+\varepsilon)}{\varepsilon}}L_x^{\frac{dp(2+\varepsilon)}{4+\varepsilon}}(I\times\R^d)}^{p-\nu}Z_I(I)^{1+\nu}.
 \label{Hj gradient}
\end{align}
Thus, we have
\begin{equation}\label{A commutator bounds}
\begin{aligned}
 &\|\nabla\mathcal A_N\|_{L_t^2L_x^{\frac{2d}{d+2}}(I\times\R^d)}+N\|\mathcal A_N\|_{L_t^2L_x^{\frac{2d}{d+2}}(I\times\R^d)}\\
 &\quad\lesssim N^{-\delta}\|u\|_{L_t^{\frac{2p(2+\varepsilon)}{\varepsilon}}L_x^{\frac{dp(2+\varepsilon)}{4+\varepsilon}}(I\times\R^d)}^{p-\nu}Z_I(I)^{1+\nu}.
\end{aligned}
\end{equation}

\medskip
\noindent\textbf{Step 2: The estimate of $\mathcal B_N$.}
By \eqref{I recursion},
\[
 I_{N_j}u=(1-\vartheta)S_{N_j}u+\vartheta I_{2N_j}u.
\]
Define
\[
 B_j:=(1-\vartheta)F(S_{N_j}u)
       +\vartheta F(I_{2N_j}u)-F(I_{N_j}u).
\]
For every integer $L\ge0$,
\begin{equation}\label{finite telescoping}
\begin{aligned}
 \sum_{j=0}^L\vartheta^jB_j
 &=\sum_{j=0}^Lc_jF(S_{N_j}u)-F(I_Nu)
       +\vartheta^{L+1}F(I_{N_{L+1}}u).
\end{aligned}
\end{equation}
For $k=0,1$, smoothness and contractivity give
\[
 \vartheta^{L+1}\|\nabla^kF(I_{N_{L+1}}u)\|_{L_t^2L_x^{\frac{2d}{d+2}}(I\times\R^d)}
 \lesssim_u\vartheta^{L+1}\longrightarrow0.
\]
Consequently,
\begin{equation}\label{B commutator series}
 \mathcal B_N=\sum_{j\ge0}\vartheta^jB_j.
\end{equation}
Recall the recursion formula of $S_N$ and $I_N$:
\begin{align*}
 S_{N_j}u&=I_{N_j}u+\vartheta(S_{N_j}-I_{2N_j})u,\\
 I_{2N_j}u&=I_{N_j}u-(1-\vartheta)(S_{N_j}-I_{2N_j})u,
\end{align*}
the fundamental theorem of calculus gives
\begin{align*}
 B_j
 ={}&\vartheta(1-\vartheta)\int_0^1(S_{N_j}-I_{2N_j})u\cdot\\
 &\quad\Bigl[F'\bigl(I_{N_j}u+\tau\vartheta(S_{N_j}-I_{2N_j})u\bigr)\\
 &\qquad-F'\bigl(I_{N_j}u-\tau(1-\vartheta)(S_{N_j}-I_{2N_j})u\bigr)
       \Bigr]\,d\tau.
\end{align*}
Hence \eqref{holder continuity} implies
\begin{equation}\label{Bj pointwise}
 |B_j|\lesssim (|S_{N_j}u|+|I_{2N_j}u|)^{p-\nu}
                         |(S_{N_j}-I_{2N_j})u|^{1+\nu}.
\end{equation}
Differentiating $B_j$ and using
$\nabla I_{N_j}u=(1-\vartheta)\nabla S_{N_j}u
                   +\vartheta\nabla I_{2N_j}u$, we obtain
\begin{align*}
 \nabla B_j
 ={}&(1-\vartheta)\nabla S_{N_j}u\cdot
             [F'(S_{N_j}u)-F'(I_{N_j}u)]\\
 &+\vartheta\nabla I_{2N_j}u\cdot
             [F'(I_{2N_j}u)-F'(I_{N_j}u)].
\end{align*}
Therefore,
\begin{equation}\label{Bj gradient pointwise}
\begin{aligned}
 |\nabla B_j|
 &\lesssim (|S_{N_j}u|+|I_{2N_j}u|)^{p-\nu}
                         |(S_{N_j}-I_{2N_j})u|^\nu\\
 &\quad\times(|\nabla S_{N_j}u|+|\nabla I_{2N_j}u|).
\end{aligned}
\end{equation}

The recursion also gives
\[
 (S_{N_j}-I_{2N_j})u
 =\vartheta^{-1}(S_{N_j}-I_{N_j})u
 =\vartheta^{-1}\bigl[(1-I_{N_j})u-(1-S_{N_j})u\bigr].
\]
By Proposition~\ref{basic property} and Sobolev embedding,
\begin{align*}
 &\|(S_{N_j}-I_{2N_j})u\|_{L_t^{2+\varepsilon}L_x^{\frac{2d(2+\varepsilon)}{d(2+\varepsilon)-4}}(I\times\R^d)}\\
 &\quad\lesssim N^{s-1}N_j^{-s}Z_I(I)
 =N^{-1}2^{-sj}Z_I(I),\\
 &\|(S_{N_j}-I_{2N_j})u\|_{L_t^{\frac{2p(2+\varepsilon)}{\varepsilon}}L_x^{\frac{dp(2+\varepsilon)}{4+\varepsilon}}(I\times\R^d)}\\
 &\quad\lesssim\||\nabla|^{s_c}(S_{N_j}-I_{2N_j})u\|_{L_t^{\frac{2p(2+\varepsilon)}{\varepsilon}}L_x^{\frac{2dp(2+\varepsilon)}{dp(2+\varepsilon)-2\varepsilon}}(I\times\R^d)}\\
 &\quad\lesssim N^{s-1}N_j^{s_c-s}Z_I(I)
 =N^{-(1-s_c)}2^{-(s-s_c)j}Z_I(I).
\end{align*}
Moreover,
\begin{equation}\label{Jensen multipliers}
\begin{aligned}
 &\|\nabla S_{N_j}u\|_{L_t^{2+\varepsilon}L_x^{\frac{2d(2+\varepsilon)}{d(2+\varepsilon)-4}}(I\times\R^d)}+\|\nabla I_{2N_j}u\|_{L_t^{2+\varepsilon}L_x^{\frac{2d(2+\varepsilon)}{d(2+\varepsilon)-4}}(I\times\R^d)}\\
 &\quad\lesssim (N_j/N)^{1-s}Z_I(I)
              =2^{(1-s)j}Z_I(I),
\end{aligned}
\end{equation}
and contractivity yields
\[
 \begin{aligned}
 &\|S_{N_j}u\|_{L_t^{\frac{2p(2+\varepsilon)}{\varepsilon}}L_x^{\frac{dp(2+\varepsilon)}{4+\varepsilon}}(I\times\R^d)}+\|I_{2N_j}u\|_{L_t^{\frac{2p(2+\varepsilon)}{\varepsilon}}L_x^{\frac{dp(2+\varepsilon)}{4+\varepsilon}}(I\times\R^d)}\\
 &\quad\le2\|u\|_{L_t^{\frac{2p(2+\varepsilon)}{\varepsilon}}L_x^{\frac{dp(2+\varepsilon)}{4+\varepsilon}}(I\times\R^d)}.
 \end{aligned}
\]
H\"older's inequality now gives
\begin{equation}\label{Bj zero}
\begin{aligned}
 &\|B_j\|_{L_t^2L_x^{\frac{2d}{d+2}}(I\times\R^d)}\\
 &\quad\lesssim\|u\|_{L_t^{\frac{2p(2+\varepsilon)}{\varepsilon}}L_x^{\frac{dp(2+\varepsilon)}{4+\varepsilon}}(I\times\R^d)}^{p-\nu}
                \|(S_{N_j}-I_{2N_j})u\|_{L_t^{\frac{2p(2+\varepsilon)}{\varepsilon}}L_x^{\frac{dp(2+\varepsilon)}{4+\varepsilon}}(I\times\R^d)}^{\nu}\\
 &\qquad\times\|(S_{N_j}-I_{2N_j})u\|_{L_t^{2+\varepsilon}L_x^{\frac{2d(2+\varepsilon)}{d(2+\varepsilon)-4}}(I\times\R^d)}\\
 &\quad\lesssim N^{-(1+\delta)}2^{-\{s+\nu(s-s_c)\}j}
                  \|u\|_{L_t^{\frac{2p(2+\varepsilon)}{\varepsilon}}L_x^{\frac{dp(2+\varepsilon)}{4+\varepsilon}}(I\times\R^d)}^{p-\nu}Z_I(I)^{1+\nu},
\end{aligned}
\end{equation}
and
\begin{equation}\label{Bj gradient}
\begin{aligned}
 &\|\nabla B_j\|_{L_t^2L_x^{\frac{2d}{d+2}}(I\times\R^d)}\\
 &\quad\lesssim\|u\|_{L_t^{\frac{2p(2+\varepsilon)}{\varepsilon}}L_x^{\frac{dp(2+\varepsilon)}{4+\varepsilon}}(I\times\R^d)}^{p-\nu}
                \|(S_{N_j}-I_{2N_j})u\|_{L_t^{\frac{2p(2+\varepsilon)}{\varepsilon}}L_x^{\frac{dp(2+\varepsilon)}{4+\varepsilon}}(I\times\R^d)}^{\nu}\\
 &\qquad\times\Bigl(\|\nabla S_{N_j}u\|_{L_t^{2+\varepsilon}L_x^{\frac{2d(2+\varepsilon)}{d(2+\varepsilon)-4}}(I\times\R^d)}
                        +\|\nabla I_{2N_j}u\|_{L_t^{2+\varepsilon}L_x^{\frac{2d(2+\varepsilon)}{d(2+\varepsilon)-4}}(I\times\R^d)}\Bigr)\\
 &\quad\lesssim N^{-\delta}2^{\{1-s-\nu(s-s_c)\}j}
                  \|u\|_{L_t^{\frac{2p(2+\varepsilon)}{\varepsilon}}L_x^{\frac{dp(2+\varepsilon)}{4+\varepsilon}}(I\times\R^d)}^{p-\nu}Z_I(I)^{1+\nu}.
\end{aligned}
\end{equation}
Hence, both series are summable and
\eqref{B commutator series} yields
\begin{equation}\label{B commutator bounds}
\begin{aligned}
 &\|\nabla\mathcal B_N\|_{L_t^2L_x^{\frac{2d}{d+2}}(I\times\R^d)}+N\|\mathcal B_N\|_{L_t^2L_x^{\frac{2d}{d+2}}(I\times\R^d)}\\
 &\quad\lesssim N^{-\delta}\|u\|_{L_t^{\frac{2p(2+\varepsilon)}{\varepsilon}}L_x^{\frac{dp(2+\varepsilon)}{4+\varepsilon}}(I\times\R^d)}^{p-\nu}Z_I(I)^{1+\nu}.
\end{aligned}
\end{equation}
Combining \eqref{commutator AB decomposition},
\eqref{A commutator bounds}, and \eqref{B commutator bounds} proves
\eqref{gradient commutator estimate}--\eqref{zero commutator estimate}.

\end{proof}

Since we work at low regularity $0<s<1$, we will need the following fractional chain rule to estimate our nonlinearity in $H_x^s$.

\begin{lemma}[Fractional chain rule $\operatorname{I}$, \cite{KV-book}]\label{F Lip}
Suppose that $F\in C^1(\mathbb C)$, $\alpha \in (0,1)$, and $1<r,r_1,r_2<\infty$ such that $\frac 1r=\frac 1{r_1}+\frac 1{r_2}$.   Then,
$$
\||\nabla|^{\alpha}F(u)\|_r\lesssim \|F'(u)\|_{r_1}\||\nabla|^{\alpha}u\|_{r_2}.
$$
\end{lemma}

When the function $F$ is no longer $C^1$, but merely H\"older continuous, we have the following useful chain rule:

\begin{lemma}[Fractional chain rule $\operatorname{II}$,\cite{KV-book}]\label{fdfp}
Let $F$ be a H\"older continuous function of order $0<\alpha<1$.  Then, for every $0<\sigma<\alpha$, $1<r<\infty$,
and $\tfrac{\sigma}{\alpha}<\rho<1$ we have
\begin{align}\label{fdfp2}
\bigl\| |\nabla|^\sigma F(u)\bigr\|_{L^r(\R^d)}
\lesssim \bigl\||u|^{\alpha-\frac{\sigma}{\rho}}\bigr\|_{L^{r_1}(\R^d)} \bigl\||\nabla|^\rho u\bigr\|^{\frac{\sigma}{\rho}}_{L^{\frac{\sigma}{\rho}r_2}(\R^d)},
\end{align}
provided $\tfrac{1}{r}=\tfrac{1}{r_1} +\tfrac{1}{r_2}$ and $(1-\frac\sigma{\alpha \rho})r_1>1$.
\end{lemma}

We also need the following consequence of the fractional chain rule,
H\"older's inequality, and Sobolev embedding; see
\cite[Lemma~2.9]{VZ07}.

\begin{lemma}[Nonlinear estimate for $M_\sigma$ norm, \cite{VZ07}]\label{use Morawetz}
Let $0<\sigma\leq s<1$ such that $\tfrac{\sigma(d-2)}{d-3+4\sigma}<s$ and let $\tfrac{4}{d}<p<\tfrac{4}{d-2s}$.  Then, there exists $\varepsilon_0>0$
sufficiently small such that on every slab $I\times\R^d$ we have
\begin{equation}\label{use Morawetz eq}
\begin{aligned}
&\|\langle\nabla\rangle^s(|u|^pu)\|_{L_t^{2}L_x^{\frac {2d}{d+2}}(I\times\R^d)}\\
&\quad\lesssim \|\langle\nabla\rangle^s u\|_{L_t^{2+\varepsilon_0}L_x^{\frac {2d(2+\varepsilon_0)}{d(2+\varepsilon_0)-4}}(I\times\R^d)}
 \|u\|^{\frac{\varepsilon_0(d-3+4\sigma)}{2\sigma(2+\varepsilon_0)}}_{M_\sigma(I\times\R^d)} \|u\|_{L_t^\infty L_x^2(I\times\R^d)}^{\alpha(\varepsilon_0)}
 \|u\|^{\beta(\varepsilon_0)}_{L_t^\infty \dot H^s_x(I\times\R^d)}.
\end{aligned}
\end{equation}
Here $\varepsilon_0$ is chosen so that $\alpha(\varepsilon_0)>0$
and $\beta(\varepsilon_0)>0$, where
$$
\alpha(\varepsilon_0):=p\bigl(1-\tfrac{d}{2s}\bigr)+\tfrac{8\sigma+\varepsilon_0[\sigma(d+2)-s(d-3+4\sigma)]}{2s\sigma(2+\varepsilon_0)}   \quad \text{and} \quad \beta(\varepsilon_0):=\tfrac{d}{s}\bigl(\tfrac{p}{2}-\tfrac{8+\varepsilon_0(d+2)}{2d(2+\varepsilon_0)}\bigr).
$$
\end{lemma}

Finally, we record the interpolation estimates used in the almost
conservation law.

\begin{lemma}\label{interpolation1}
Let $d\ge3$, $0<\sigma<1$, and $\frac4d<p<\frac4{d-2}$.
Define $\theta$ by
\begin{align}\label{parameter}
\theta=\begin{cases}
\dfrac{(1-s_c)(d-3+4\sigma)}{d-3+6\sigma-d\sigma},
 &\mbox{if }s_c\ge\dfrac{\sigma(d-2)}{d-3+4\sigma},\\[6pt]
\dfrac{s_c(d-3+4\sigma)}{\sigma(d-2)},
 &\mbox{if }s_c<\dfrac{\sigma(d-2)}{d-3+4\sigma},
\end{cases}
\end{align}
and choose
\[
\varepsilon=\frac{4p\sigma\theta}{d-3+4\sigma-2p\sigma\theta}>0.
\]
Then $0<\theta\le1$ and, on every time interval $I\subseteq\R$,
\begin{align}
&\|u\|_{L_t^{\frac{2p(2+\varepsilon)}{\varepsilon}}
 L_x^{\frac{dp(2+\varepsilon)}{4+\varepsilon}}(I\times\R^d)}\notag\\
&\quad\lesssim\|u\|_{M_\sigma(I\times\R^d)}^\theta
 \|\nabla u\|_{L_t^\infty L_x^2(I\times\R^d)}^{1-\theta},
 \qquad s_c\ge\frac{\sigma(d-2)}{d-3+4\sigma},
 \label{interpolation high power}\\
&\|u\|_{L_t^{\frac{2p(2+\varepsilon)}{\varepsilon}}
 L_x^{\frac{dp(2+\varepsilon)}{4+\varepsilon}}(I\times\R^d)}\notag\\
&\quad\le\|u\|_{M_\sigma(I\times\R^d)}^\theta
 \|u\|_{L_t^\infty L_x^2(I\times\R^d)}^{1-\theta},
 \qquad s_c\le\frac{\sigma(d-2)}{d-3+4\sigma}.
 \label{interpolation low power}
\end{align}
The implicit constant is independent of $I$.
\end{lemma}
\begin{proof}
Recall that $s_c=\frac d2-\frac2p\in(0,1)$. Since
$
0<\frac{\sigma(d-2)}{d-3+4\sigma}<1,
$
the definition of $\theta$ gives $0<\theta\le1$.
We will prove the two estimates by interpolation. First, we consider the region $s_c\geq\frac{\sigma(d-2)}{d-3+4\sigma}$. For the time component, we need that 
\begin{align*}
    \frac{\varepsilon}{2p(2+\varepsilon)}=\frac{\sigma\theta}{d-3+4\sigma},
\end{align*}
which implies that $\varepsilon=\frac{\sigma\theta}{d-3+4\sigma-2p\sigma\theta}>0$. Next, we consider the spatial component. 
\begin{align}\label{interpolation-1}
    \frac{4+\varepsilon}{pd(2+\varepsilon)}=\frac{\theta(d-3+2\sigma)}{2(d-3+4\sigma)}+\frac{d-2}{2d}(1-\theta).
\end{align}
On the other hand, we have
\begin{align}\label{interpolation-2}
    \frac{4+\varepsilon}{pd(2+\varepsilon)}=\frac{1}{2}-\frac{s_c}{d}-\frac{ 2\sigma\theta}{d(d-3+4\sigma)}.
\end{align}
Putting \eqref{interpolation-1} and \eqref{interpolation-2} together, we have
\begin{align*}
    1-s_c=\theta\Big(1-\frac{\sigma(d-2)}{d-3+4\sigma}\Big),
\end{align*}
which is equivalent to $s_c\geq\frac{\sigma(d-2)}{d-3+4\sigma}$. 
It then follows from the H\"older inequality and Soblev embedding,
\begin{align*}
\|u\|_{L_t^{\frac{2p(2+\varepsilon)}{\varepsilon}}
 L_x^{\frac{dp(2+\varepsilon)}{4+\varepsilon}}(I\times\R^d)}&\le
 \|u\|_{L_t^{\frac{d-3+4\sigma}{\sigma}}
 L_x^{\frac{2(d-3+4\sigma)}{d-3+2\sigma}}(I\times\R^d)}^\theta
 \|u\|_{L_t^\infty L_x^{\frac{2d}{d-2}}(I\times\R^d)}^{1-\theta}\\
&\lesssim\|u\|_{M_\sigma(I\times\R^d)}^\theta
 \|\nabla u\|_{L_t^\infty L_x^2(I\times\R^d)}^{1-\theta}.
\end{align*}

If $s_c<\frac{\sigma(d-2)}{d-3+4\sigma}$, the similar calculation yields
$s_c=\frac{\theta\sigma(d-2)}{d-3+4\sigma}$, and consequently
\begin{align*}
\frac{4+\varepsilon}{dp(2+\varepsilon)}
&=\frac12-\frac{\sigma\theta}{d-3+4\sigma}=\frac{\theta(d-3+2\sigma)}{2(d-3+4\sigma)}
 +\frac{1-\theta}{2}.
\end{align*}
By the H\"older inequality, we have
\begin{align*}
\|u\|_{L_t^{\frac{2p(2+\varepsilon)}{\varepsilon}}
L_x^{\frac{dp(2+\varepsilon)}{4+\varepsilon}}(I\times\R^d)}&\le
 \|u\|_{L_t^{\frac{d-3+4\sigma}{\sigma}}
 L_x^{\frac{2(d-3+4\sigma)}{d-3+2\sigma}}(I\times\R^d)}^\theta
 \|u\|_{L_t^\infty L_x^2(I\times\R^d)}^{1-\theta}\\
&=\|u\|_{M_\sigma(I\times\R^d)}^\theta
 \|u\|_{L_t^\infty L_x^2(I\times\R^d)}^{1-\theta}.
\end{align*}
\end{proof}

\section{Proof of Theorem~\ref{scattering}}
\subsection{Global well-posedness}

We prove Theorem~\ref{scattering} using an almost conservation law
for $E(I_Nu)$. The function $I_Nu$ does not solve \eqref{equation},
so the energy conservation law for that equation does not apply
directly to $E(I_Nu)$. The following proposition controls its increment.

\begin{proposition}\label{energy increment scattering}
Let $d\ge3$, $\frac4d<p<\frac4{d-2}$, $s_c<s<1$, and $0<\sigma<1$.
There exist $\eta>0$ and $N_0\ge1$, depending only on $d,p,s,\sigma$,
with the following property. Let $N\ge N_0$ and let
$u\in C(J;H^s(\R^d))$ be a solution to \eqref{equation}
on $J=[t_0,T]$ with $\|\nabla I_Nu(t_0)\|_{L_x^2}\le1$.
Suppose that
\begin{align}\label{mass independent smallness}
\|u\|_{M_\sigma(J\times\R^d)}
\le\eta(1+\|u(t_0)\|_{L_x^2})^{-b/s_c},
\end{align}
where $b=(\frac{\sigma(d-2)}{d-3+4\sigma}-s_c)_+$.
Then the norm $Z_I(J)$ defined in \eqref{ZI} satisfies
$Z_I(J)\lesssim1$ and
\begin{equation}\label{new energy increment}
 \sup_{t\in J}|E(I_Nu(t))-E(I_Nu(t_0))|
 \lesssim N^{-\delta},
\end{equation}
where $\delta=\nu(1-s_c)$.
\end{proposition}

Therefore, with Proposition \ref{energy increment scattering} at place, the proof of global well-posedness has been reduced to showing
\begin{align}\label{global Morawetz}
\|u\|_{M_\sigma(\R\times\R^d)}\leq C(\|u_0\|_{H_x^s}).
\end{align}
This also implies scattering, as we will show below.

Recall that interpolating between the interaction Morawetz inequality 
\begin{align*}
    \big\||\nabla|^{-\frac{d-3}{4}}u\big\|_{L_{t,x}^4(I\times\R^d)}\lesssim \|u_0\|_{L^2}^\frac12\|u\|_{L_t^\infty\dot H^\frac12(I\times\R^d)}^\frac12
\end{align*}
and $L_t^\infty \dot H_x^\sigma(I\times\R^d)$,
$0<\sigma\leq s$, we get
\begin{equation}\label{inter Mora esti}
\|u\|_{M_\sigma(I\times\R^d)}\lesssim \bigl(\|u\|_{L_t^\infty L_x^2(I\times\R^d)}\|u\|_{L_t^\infty\dot H_x^{\frac12}(I\times\R^d)}\bigr)^{\frac {2\sigma}{d-3+4\sigma}}\|u\|_{L_t^\infty\dot H^\sigma_x (I\times\R^d)}^{\frac{d-3}{d-3+4\sigma}}
\end{equation}
on any spacetime slab $I\times\R^d$ on which the solution to \eqref{equation} exists and lies in $H_x^{\max\{\frac 12,\sigma\}}$.
However, the $H_x^{\max\{\frac 12,\sigma\}}$-norm of the solution is not conserved. To control it, we must resort the $H^s_x$ bound on the solution.  Thus, in order to obtain a global Morawetz estimate we need a global $H^s_x$ bound, which 
will proceed by a bootstrap argument.

Let $u(x,t)$ be the solution to \eqref{equation}.  As $E(I_Nu_0)$ is not necessarily small, we will rescale the solution such that the energy of the
rescaled initial data satisfies the hypothesis of Proposition~\ref{energy increment scattering}.  Indeed, by scaling,
$$
u^{\lambda}(x,t):=\lambda^{-\frac 2p}u\Bigl(\frac x{\lambda},\frac t{\lambda^2}\Bigr)
$$
is a solution to \eqref{equation} with initial data
$$
u_0^{\lambda}:=\lambda^{-\frac 2p}u_0\Bigl(\frac x{\lambda}\Bigr).
$$
By Proposition~\ref{basic property} and Sobolev embedding,
\begin{align*}
&\|\nabla I_Nu_0^\lambda\|_{L_x^2}\lesssim N^{1-s}\|u_0^\lambda\|_{\dot H^s_x}=N^{1-s}\lambda^{s_c-s} \|u_0\|_{\dot H^s_x}.
\end{align*}
Choose
\begin{equation}\label{initial scaling choice}
 \lambda=K N^{\frac{1-s}{s-s_c}},\qquad K\ge1.
\end{equation}
The right side is $K^{s_c-s}\|u_0\|_{\dot H_x^s}$, which is small
when $K$ is sufficiently large depending on the initial data.
If $s\le\frac{dp}{2(2+p)}$, the symbol comparison gives
\[
 |\xi|^\frac{dp}{2(2+p)} m_N(\xi)\lesssim N^{\frac{dp}{2(2+p)}-s}|\xi|^s.
\]
Indeed, below $N$ we use $|\xi|^{\frac{dp}{2(2+p)}-s}\le N^{\frac{dp}{2(2+p)}-s}$.
Above $N$ we use $|\xi|^{\frac{dp}{2(2+p)}-1}\le N^{\frac{dp}{2(2+p)}-1}$.
Consequently Sobolev embedding yields
\begin{align}\label{modified potential bound}
 \|I_Nu_0^\lambda\|_{L_x^{p+2}}
 &\lesssim\||\nabla|^\frac{dp}{2(2+p)} I_Nu_0^\lambda\|_{L_x^2}
 \lesssim N^{\frac{dp}{2(2+p)}-s}\lambda^{s_c-s}\|u_0\|_{\dot H^s}\notag\\
 &=K^{s_c-s}N^{\frac{dp}{2(2+p)}-1}\|u_0\|_{\dot H^s}\longrightarrow0.
\end{align}
If $s>\frac{dp}{2(p+2)}$, by the $L^p$-boundedness of $I_N$ and Sobolev embedding:
\[
 \|I_Nu_0^\lambda\|_{L_x^{p+2}}
 \le\|u_0^\lambda\|_{L_x^{p+2}}
 \lesssim\lambda^{s_c-\frac{dp}{2(2+p)}}\|u_0\|_{H^s}\longrightarrow0.
\]
Thus, after fixing $K$ and increasing $N$,
\begin{equation}\label{small initial modified energy}
 E(I_Nu_0^\lambda)\le\tfrac18.
\end{equation}

We now show that there exists an absolute constant $C_1$ such that
\begin{align}\label{rescaled Mora}
\|u^{\lambda}\|_{M_\sigma(\R\times\R^d)}\le C_1\lambda^{\frac {s_c[d-3-\sigma(d-6)]}{d-3+4\sigma}}.
\end{align}
Undoing the scaling, this yields \eqref{global Morawetz}.

By time reversal symmetry, it suffices to argue for positive times only.  Define
\begin{equation*}
\Omega_1:=\{t\in[0,\infty): \ \|u^{\lambda}\|_{M_\sigma([0,t]\times\R^d)}\le C_1\lambda^{\frac {s_c[d-3-\sigma(d-6)]}{d-3+4\sigma}}\}.
\end{equation*}
We want to show that $\Omega_1=[0,\infty)$.  We achieve this via a bootstrap argument.  Let
$$
\Omega_2:=\{t\in[0,\infty): \ \|u^{\lambda}\|_{M_\sigma([0,t]\times\R^d)}\le 2C_1\lambda^{\frac {s_c[d-3-\sigma(d-6)]}{d-3+4\sigma}}\}.
$$
In order to run the bootstrap argument successfully, we need to verify four things:
\begin{itemize}
    \item[(1)] $\Omega_1$ is nonempty,
    \item[(2)]
$\Omega_1$ is closed
\item[(3)] $\Omega_2\subset\Omega_1$,
    \item[(4)]  If $T\in \Omega_1$, then there exists $\eps>0$ such that $[T,T+\eps)\subset\Omega_2$.
\end{itemize}
The claims $(1)$ and $(2)$  are trivial. 
We now show $(3)$.  Let $T\in \Omega_2$, we will show that $T\in \Omega_1$.  By \eqref{inter Mora esti} and \eqref{mass}, we obtain
\begin{align}
\|u^{\lambda}\|_{M_\sigma([0,T]\times\R^d)}
&\lesssim \bigl(\|u_0^\lambda\|_{L_x^2}\|u^\lambda\|_{L_t^\infty\dot H_x^{\frac12}([0,T]\times\R^d)}\bigr)^{\frac {2\sigma}{d-3+4\sigma}}\|u^\lambda\|_{L_t^\infty\dot H^\sigma_x ([0,T]\times\R^d)}^{\frac{d-3}{d-3+4\sigma}}\nonumber\\
&\lesssim C(\|u_0\|_{L_x^2})\lambda^{\frac {2\sigma s_c}{d-3+4\sigma}}\|u^{\lambda}\|_{L_t^\infty\dot H_x^{\frac12}([0,T]\times\R^d)}^{\frac {2\sigma}{d-3+4\sigma}}\|u^\lambda\|_{L_t^\infty\dot H^\sigma_x ([0,T]\times\R^d)}^{\frac{d-3}{d-3+4\sigma}}.\label{Morawetz on 0 T}
\end{align}
To control the last two factors, we decompose $u^\lambda(t)$ into
$$
u^{\lambda}(t):=P_{\le N}u^{\lambda}(t)+P_{>N}u^{\lambda}(t).
$$
To estimate the low frequencies, we interpolate between the $L_x^2$-norm
and $\dot H_x^1$-norm and use 
\eqref{symbol comparison}:
\begin{align}
\|P_{\le N}u^{\lambda}(t)\|_{\dot H^\sigma_x}
&\lesssim \|P_{\le N}u^{\lambda}(t)\|_{L_x^2}^{1-\sigma} \|P_{\le N}u^{\lambda}(t)\|_{\dot H_x^1}^\sigma\notag\\
&\lesssim \lambda^{s_c(1-\sigma)}C(\|u_0\|_{L_x^2})\|Iu^{\lambda}(t)\|_{\dot H_x^1}^\sigma\\
\|P_{\le N}u^{\lambda}(t)\|_{\dot H^{\frac 12}_x}
&\lesssim \|P_{\le N}u^{\lambda}(t)\|_{L_x^2}^{\frac 12} \|P_{\le N}u^{\lambda}(t)\|_{\dot H_x^1}^{\frac 12}\notag\\
&\lesssim \lambda^{\frac {s_c}2}C(\|u_0\|_{L_x^2})\|I_Nu^{\lambda}(t)\|_{\dot H_x^1}^{\frac 12}.\label{low}
\end{align}
For the high frequency part, by  interpolating between $L_x^2$ and $\dot H^s_x$, we get
\begin{align}
\|P_{> N}u^{\lambda}(t)\|_{\dot H_x^\sigma}
&\lesssim \|P_{> N}u^{\lambda}(t)\|_{L_x^2}^{1-\frac \sigma{s}}\|P_{> N}u^{\lambda}(t)\|_{\dot H^s_x}^{\frac \sigma{s}} \notag\\
&\lesssim \lambda^{s_c(1-\frac{\sigma}{s})}\|u_0\|_{L_x^2}^{1-\frac{\sigma}{s}}N^{\frac{\sigma(s-1)}{s}}\|I_N u^{\lambda}(t)\|_{\dot H_x^1}^{\frac \sigma{s}}\notag\\
&\lesssim \lambda^{s_c(1-\sigma)}C(\|u_0\|_{L_x^2})\|I_Nu^{\lambda}(t)\|_{\dot H_x^1}^{\frac \sigma{s}}\\
\|P_{> N}u^{\lambda}(t)\|_{\dot H_x^{\frac 12}}
&\lesssim \|P_{> N}u^{\lambda}(t)\|_{L_x^2}^{1-\frac 1{2s}}\|P_{> N}u^{\lambda}(t)\|_{\dot H^s_x}^{\frac 1{2s}} \notag\\
&\lesssim \lambda^{(1-\frac 1{2s})s_c}N^{\frac {s-1}{2s}}\|u_0\|_{L_x^2}^{1-\frac 1{2s}}\|I_Nu^{\lambda}(t)\|_{\dot H_x^1}^{\frac 1{2s}}\nonumber\\
&\lesssim \lambda^{\frac {s_c}2}C(\|u_0\|_{L_x^2})\|I_Nu^{\lambda}(t)\|_{\dot H_x^1}^{\frac 1{2s}}.\label{high}
\end{align}
Putting these estimates together, we obtain
\begin{align}
\|u^{\lambda}\|_{M_\sigma([0,T]\times\R^d)}\lesssim C(&\|u_0\|_{L_x^2})\lambda^{\frac {s_c[d-3-\sigma(d-6)]}{d-3+4\sigma}}\label{u lambda Morawetz}\notag\\
&\times\sup_{[0,T]}\bigl(\|\nabla I_N u^{\lambda}(t)\|_{L_x^2}^{\frac 12}+\|\nabla I_N u^{\lambda}(t)\|_{L_x^2}^{\frac 1{2s}}\bigr)^{\frac {2\sigma}{d-3+4\sigma}}\notag\\
&\times\sup_{[0,T]}\bigl(\|\nabla I_N u^{\lambda}(t)\|_{L_x^2}^\sigma+\|\nabla I_N u^{\lambda}(t)\|_{L_x^2}^{\frac{\sigma}{s}}\bigr)^{\frac {d-3}{d-3+4\sigma}}.
\end{align}
Thus, taking $C_1$ sufficiently large depending on $\|u_0\|_{L_x^2}$, as a consequence of the following claim,\begin{equation}\label{bdd kinetic}
\sup_{[0,T]}\|\nabla I_N u^{\lambda}(t)\|_{L_x^2}\le 1.
\end{equation} we get $T\in \Omega_1$.

 We next show that $T\in \Omega_2$ yields \eqref{bdd kinetic}.  By Proposition \ref{energy increment scattering}, we take the rescale parameter $$\eta_\lambda=\eta(1+\lambda^{s_c}\|u_0\|_{L^2})^{-\frac{b}{s_c}}$$ where $b=(\frac{\sigma(d-2)}{d-3+4\sigma}-s_c)_+$. Splitting the time interval $[0,T]$ into several subintervals $I_j=[t_j,t_{j+1}]$ on which $\|u^\lambda\|_{M_\sigma(I_j\times\R^d)}\leq\eta_\lambda$. Mass conservation gives
$\|u^\lambda(t_j)\|_{L_x^2}=\lambda^{s_c}\|u_0\|_{L_x^2}$, so each $I_j$
satisfies \eqref{mass independent smallness} for $u^\lambda$.
The number of such subintervals can be bounded by
\begin{align*}
  L\leq 1+C\eta_\lambda^{-\frac{d-3+4\sigma}{\sigma}}\lambda^{\frac{s_cA_d(\sigma)}{\sigma}}\leq C(d,p,s,\|u_0\|_{L^2})\lambda^{\max\{\frac{s_cA_d(\sigma)}{\sigma},(d-2)(1-s_c)\}}.
\end{align*}
Next, utilizing Proposition \ref{energy increment scattering} on each subinterval $I_j$ again, we obtain
\begin{align}\label{energy-increment-3}
  \sup_{0\leq t\leq T}E(I_Nu^\lambda(t))\leq E(I_Nu_0^\lambda)+C(E(I_Nu_0^\lambda))LN^{-\delta}.
\end{align}
Taking $\lambda\sim N^{\frac{1-s}{s-s_c}}$, the second term in
\eqref{energy-increment-3} tends to zero as $N\to\infty$ provided
\[
\frac{\kappa(1-s)}{s-s_c}-\delta<0,
\qquad
\kappa=\max\left\{\frac{s_cA_d(\sigma)}{\sigma},(d-2)(1-s_c)\right\}.
\]
Choosing $\sigma=s$ and using $s>s_c>0$, this condition is equivalent to
\begin{align*}
  s_c(1-s)A_d(s)<\delta s(s-s_c), \\
(d-2)(1-s_c)(1-s)<\delta(s-s_c).
\end{align*}
By \eqref{rho definition} and the monotonicity used to define $\rho_1$,
the first inequality is equivalent to $s>\rho_1(d,p)$.
Since $\delta=\nu(1-s_c)$, the second inequality is equivalent to
\[
(d-2)(1-s)<\nu(s-s_c)
\quad\Longleftrightarrow\quad
s>\frac{d-2+\nu s_c}{d-2+\nu}=\rho_2(d,p).
\]
Thus both inequalities hold precisely when $s>s_*(d,p)$, as required
in \eqref{main threshold}.
By using the bootstrap assumption and induction, we obtain
$E(I_Nu^\lambda(t))\leq\frac14$ for $t\in[0,T]$.
In particular, the hypothesis on
$\|\nabla I_Nu^\lambda(t_j)\|_{L_x^2}$ is satisfied at every step,
and hence \eqref{bdd kinetic} holds.

At time $t=T$, $u^\lambda(T)\in H^s(\R^d)$ and enjoys the bound \eqref{bdd kinetic}. By the local well-posedness, it can be extended to the interval $[T,T+\varepsilon)$ for $\varepsilon>0$ sufficiently small. By the continuity, it also holds $$\sup_{t\in[T,T+\varepsilon)}\|\nabla I_Nu^\lambda\|_{L^2(\R^d)}\leq1.$$
Since $T\in\Omega_1$, we have $$\|u^{\lambda}\|_{M_\sigma([0,T]\times\R^d)}\le C_1\lambda^{\frac {s_c[d-3-\sigma(d-6)]}{d-3+4\sigma}}.$$
On the other hand, there exists $\eta>0$ sufficiently small such that 
$\|u^\lambda\|_{M_\sigma([T,T+\varepsilon)\times\R^d}<\eta$, then we have
\begin{align*}
    \|u^{\lambda}\|_{M_\sigma([0,T+\varepsilon)\times\R^d)}\le 2C_1\lambda^{\frac {s_c[d-3-\sigma(d-6)]}{d-3+4\sigma}},
\end{align*}
which means that $[T,T+\varepsilon)\in\Omega_2$.

Thus, the bootstrap is complete and we have \eqref{rescaled Mora}.  Hence, \eqref{bdd kinetic} holds for all $T\in \R$.  By using \eqref{i3}
and the conservation of mass, it  implies
\begin{align*}
\|u(T)\|_{H_x^s}
&\lesssim \|u_0\|_{L_x^2} + \|u(T)\|_{\dot H^s_x}\\
&\lesssim \|u_0\|_{L_x^2} + \lambda^{s-s_c}\|u^\lambda(\lambda^2T)\|_{\dot H^s_x}\\
&\lesssim \|u_0\|_{L_x^2} + \lambda^{s-s_c}\|Iu^\lambda(\lambda^2T)\|_{H^1_x}\\
&\lesssim \|u_0\|_{L_x^2} + \lambda^{s-s_c}\bigl(\|u^\lambda(\lambda^2T)\|_{L^2_x}+\|\nabla I_Nu^\lambda(\lambda^2T)\|_{L_x^2}\bigr)\\
&\lesssim \|u_0\|_{L_x^2} + \lambda^{s-s_c}(\lambda^{s_c}\|u_0\|_{L_x^2}+1)\\
&\leq C(\|u_0\|_{H^s_x}),
\end{align*}
for all $T\in \R$.  Therefore,
\begin{align}\label{global H^s}
\|u\|_{L_t^\infty H^s_x(\R\times\R^d)}\leq C(\|u_0\|_{H^s_x}).
\end{align}

\subsection{Scattering}For completeness, we give the detailed proof of the scattering from the global bound in $M_\sigma$. We proceed by following the argument in Vi\c{s}an and Zhang \cite{VZ07}.
We first show that the global space-time estimate \eqref{global Morawetz} can be upgraded to the global Strichartz bound
\begin{equation}\label{w bound}
\|u\|_{W(\R)}:=\sup_{(q,r)\ {\rm admissible}}\|\langle\nabla\rangle^s u\|_{L_t^qL_x^r(\R\times\R^d)}\leq C(\|u_0\|_{H^s_x}).
\end{equation}
The second step is to use this estimate to prove asymptotic completeness of the wave operator. The existence of the wave operators is rather standard and we omit it.

Let $u$ be a global solution to \eqref{equation} under the hypotheses of
Theorem~\ref{scattering}.  By \eqref{global Morawetz} we have
$$
\|u\|_{M_\sigma(\R\times\R^d)}\le C(\|u_0\|_{H^s_x}).
$$
Let $\eta_0>0$ be a small constant to be chosen momentarily and split $\R$ into $L\le 1+C(\|u_0\|_{H^s_x})\eta_0^{-(d-3+4\sigma)/\sigma}$ subintervals $I_j=[t_j,t_{j+1}]$ such that
$$
\|u\|_{M_\sigma(I_j\times\R^d)}\le \eta_0.
$$
By Strichartz,
\begin{equation}\label{wuij}
\|u\|_{W(I_j)}
\lesssim\|\langle\nabla\rangle^su(t_j)\|_{L_x^2}+\|\langle\nabla\rangle^s(|u|^pu)\|_{L_t^2L_x^{\frac {2d}{d+2}}(I_j\times\R^d)}.
\end{equation}
We take $\sigma=s$. Since $s>\frac12$ and $s>s_c$, the hypotheses
of Lemma~\ref{use Morawetz} hold. Choose $\varepsilon_0>0$ as in that
lemma, independently of the parameter used in Proposition~\ref{energy increment scattering}.
Using \eqref{global H^s}, we control the nonlinearity as follows:
\begin{align}
\|\langle\nabla\rangle^s(|u|^pu)\|_{L_t^2L_x^{\frac {2d}{d+2}}(I_j\times\R^d)}
&\lesssim \|u\|_{W(I_j)}\|u\|_{M_\sigma(I_j\times \R^d)}^{\frac{\varepsilon_0(d-3+4\sigma)}{2\sigma(2+\varepsilon_0)}}\|u\|^{\alpha(\varepsilon_0)+\beta(\varepsilon_0)}_{L_t^{\infty} H_x^s(I_j\times \R^d)}\notag\\
&\lesssim \|u\|_{W(I_j)}\eta_0^{\frac{\varepsilon_0(d-3+4\sigma)}{2\sigma(2+\varepsilon_0)}} C(\|u_0\|_{H^s_x}).\label{nonlinearity}
\end{align}
Taking $\eta_0$ sufficiently small depending only on $\|u_0\|_{H^s_x}$, \eqref{wuij} and \eqref{nonlinearity} yield
$$
\|u\|_{W(I_j)}\lesssim \|\langle\nabla\rangle^su(t_j)\|_{L^2}.
$$
Adding these bounds over all subintervals $I_j$, we obtain \eqref{w bound}.

We now use \eqref{w bound} to show asymptotic completeness, i.e., there exist unique $u_{\pm}\in H^s_x$ such that
$$
\lim_{t\to \pm\infty}\|u(t)-e^{it\Delta}u_{\pm}\|_{H^s_x}=0.
$$
By time reversal symmetry, it suffices to argue in the positive time direction.  For $t>0$ define $v(t) = e^{-it\Delta}u(t)$. We will show
that $v(t)$ converges in $H^s_x$ as $t\rightarrow \infty$, and define $u_+$ to be that limit.

Indeed, from Duhamel's formula \eqref{duhamel} we have
\begin{align}\label{v}
v(t) = u_0 - i\int_{0}^{t} e^{-is\Delta}(|u|^pu)(s)ds.
\end{align}
Therefore, for $0<\tau<t$,
$$
v(t)-v(\tau)=-i\int_{\tau}^{t}e^{-is\Delta}(|u|^pu)(s)ds.
$$
By Strichartz and Lemma~\ref{use Morawetz}, we estimate
\begin{align*}
\|v(t)-v(\tau)\|_{H^s_x}
&=\|e^{it\Delta}[v(t)-v(\tau)]\|_{H^s_x} \\
&\lesssim \|\langle\nabla\rangle^s(|u|^pu)\|_{L_t^2L_x^{\frac {2d}{d+2}}([\tau,t]\times\R^d)}\\
&\lesssim \|u\|_{W([\tau,t])}\|u\|_{M_\sigma([\tau,t]\times\R^d)}^{\frac{\varepsilon_0(d-3+4\sigma)}{2\sigma(2+\varepsilon_0)}}\|u\|^{\alpha(\varepsilon_0)+\beta(\varepsilon_0)}_{L_t^{\infty} H_x^s([\tau,t]\times\R^d)}.
\end{align*}
Using \eqref{global Morawetz}, \eqref{global H^s}, and \eqref{w bound}, we obtain
\begin{center}
$\|v(t)-v(\tau)\|_{H^s_x}\rightarrow 0 \quad$ as $\tau,t \rightarrow \infty$.
\end{center}
In particular, this Cauchy sequence implies that $u_{+}$ is well defined. Also, inspecting \eqref{v} one easily sees that
\begin{align}
u_{+}=u_0- i\int_{0}^{\infty}e^{-is\Delta}(|u|^pu)(s)ds
\end{align}
and thus
\begin{align}\label{u+}
e^{it\Delta}u_{+}=e^{it\Delta}u_0- i\int_{0}^{\infty}e^{i(t-s)\Delta}(|u|^pu)(s)ds.
\end{align}
By the same arguments as above, \eqref{u+} and Duhamel's formula \eqref{duhamel} imply that
$\|u(t)-e^{it\Delta}u_{+}\|_{H^s_x}\rightarrow 0$ as $t\rightarrow\infty$.

\subsection{Proof of Proposition~\ref{energy increment scattering}}Now, we give the proof of Proposition \ref{energy increment scattering}. For convenience, we write $w:=I_Nu$ and the commutator
\begin{equation}\label{commutator definition}
\mathcal C_N:=I_NF(u)-F(w).
\end{equation}
Then, $w$ satisfies the following equation
\begin{equation}\label{modified equation}
iw_t+\Delta w=F(w)+\mathcal C_N.
\end{equation}
Throughout this proof, $J=[t_0,T]$ and $Z_I(J)$ denotes the norm in \eqref{ZI} evaluated on $J\times\R^d$.
Using Lemma \ref{interpolation1} and the high-low decomposition as in \cite{VZ07}, we claim that 
\begin{equation}\label{estimate of u}
\|u\|_{L_t^{\frac{2p(2+\varepsilon)}{\varepsilon}}L_x^{\frac{dp(2+\varepsilon)}{4+\varepsilon}}(J\times\R^d)}\lesssim\eta^\theta(1+Z_I(J))^{1-\theta}
 +N^{s_c-1}Z_I(J).
\end{equation}
In fact, if $s_c\ge \frac{\sigma(d-2)}{d-3+4\sigma}$, we 
apply \eqref{interpolation high power} to $P_{\le N}u$
\[
\begin{aligned}
\|P_{\leq N}u\|_{L_t^{\frac{2p(2+\varepsilon)}{\varepsilon}}L_x^{\frac{dp(2+\varepsilon)}{4+\varepsilon}}(J\times\R^d)}&\lesssim\|P_{\leq N}u\|_{M_\sigma(J\times\R^d)}\|\nabla P_{\leq N}u\|_{L_t^\infty L_x^2(J\times\R^d)}\\
&\lesssim\eta^\theta Z_I(J)^{1-\theta}.
\end{aligned}
\]
For the high frequency part $P_{>N}u$,  Sobolev embedding and \eqref{i2} with
$\gamma=s_c<s$ give
\[
 \begin{aligned}
 &\|P_{>N}u\|_{L_t^{\frac{2p(2+\varepsilon)}{\varepsilon}}L_x^{\frac{dp(2+\varepsilon)}{4+\varepsilon}}(J\times\R^d)}\\
&\quad\lesssim\||\nabla|^{s_c}P_{>N}u\|_{L_t^{\frac{2p(2+\varepsilon)}{\varepsilon}}L_x^{\frac{2dp(2+\varepsilon)}{dp(2+\varepsilon)-2\varepsilon}}(J\times\R^d)}\lesssim N^{s_c-1}Z_I(J).
 \end{aligned}
\]
If $s_c< \frac{\sigma(d-2)}{d-3+4\sigma}$, we  apply \eqref{interpolation low power} directly to
$u$. The mass conservation and \eqref{mass independent smallness} provides that
\begin{align*}
  \|u\|_{L_t^{\frac{2p(2+\varepsilon)}{\varepsilon}}L_x^{\frac{dp(2+\varepsilon)}{4+\varepsilon}}(J\times\R^d)}\leq\eta^\theta(1+\|u(t_0)\|_{2})^{-(1-\theta)}\|u(t_0)\|_{2}^{1-\theta}\leq\eta^\theta.
\end{align*}
This proves \eqref{estimate of u} in both regions. 
Since $I_N$ is a contraction on $L_t^{\frac{2p(2+\varepsilon)}{\varepsilon}}L_x^{\frac{dp(2+\varepsilon)}{4+\varepsilon}}(J\times\R^d)$, the H\"older inequality together with the  Strichartz estimate and
Lemma~\ref{heat commutator} yield
\begin{align}
 Z_I(J)\lesssim{}&\|\nabla I_Nu(t_0)\|_2
 +\|u\|_{L_t^{\frac{2p(2+\varepsilon)}{\varepsilon}}L_x^{\frac{dp(2+\varepsilon)}{4+\varepsilon}}(J\times\R^d)}^pZ_I(J)\notag\\
 &+N^{-\delta}\|u\|_{L_t^{\frac{2p(2+\varepsilon)}{\varepsilon}}L_x^{\frac{dp(2+\varepsilon)}{4+\varepsilon}}(J\times\R^d)}^{p-\nu}Z_I(J)^{1+\nu}.
 \label{Z bootstrap}
\end{align}
The same estimates hold on every initial subinterval $J'=[t_0,t]\subseteq J$. Now, we set $L_0$ to be the constant of Strichartz estimate and the bootstrap estimate $Z_I(J)\leq 4C_0$. Taking $\eta$ to be small enough and then $N_0$ large which is depending on $d,p,s$ such that \eqref{estimate of u} makes the nonlinear terms in \eqref{Z bootstrap} at most $C_0$ in total. The bootstrap argument therefore gives $Z_I(J)\lesssim1$.  In particular,
\begin{equation}\label{local commutator consequence}
 \|\nabla\mathcal C_N\|_{L_t^{2}L_x^{\frac{2d}{d+2}}(J\times\R^d)}
 \lesssim N^{-\delta}.
\end{equation}

We also obtain, from $I_NF(u)=F(w)+\mathcal C_N$,
\[
\|\nabla I_NF(u)\|_{L_t^2L_x^{\frac{2d}{d+2}}(J\times\R^d)}
\lesssim1.
\]

We now compute the energy increment. For smooth solutions,
\eqref{modified equation} gives
\begin{equation}\label{clean energy identity}
 E(w(T))-E(w(t_0))
 =-\Im\int_{t_0}^{T}\!\int_{\R^d}
 \{\nabla w\cdot\nabla\overline{\mathcal C_N}
       +F(w)\overline{\mathcal C_N}\}\,dx\,dt,
\end{equation}
where $F(w)=I_NF(u)-\mathcal C_N$. Notice that 
$\Im\int|\mathcal C_N|^2=0$, the second term in
\eqref{clean energy identity} can be written as $I_NF(u)$.
Thus, by H\"older's inequality and Sobolev embedding, it follows 
\begin{equation}\label{energy Holder}
\begin{aligned}
 &|E(w(T))-E(w(t_0))|\\
 &\quad\lesssim Z_I(J)\|\nabla\mathcal C_N\|_{L_t^{2}L_x^{\frac{2d}{d+2}}(J\times\R^d)}\\
 &\qquad+\||\nabla|^{-1}I_NF(u)\|_{L_t^{2}L_x^{\frac{2d}{d-2}}(J\times\R^d)}
 \|\nabla\mathcal C_N\|_{L_t^{2}L_x^{\frac{2d}{d+2}}(J\times\R^d)}.
\end{aligned}
\end{equation}
The Hardy--Littlewood--Sobolev inequality  together
with the Riesz transforms, gives
\[
\||\nabla|^{-1}f\|_{L^{\frac{2d}{d-2}}}
=\||\nabla|^{-2}(|\nabla|f)\|_{L^{\frac{2d}{d-2}}}
\lesssim\|\nabla f\|_{L^{\frac{2d}{d+2}}}.
\]
Applying this estimate at each time and using the preceding bound yields
\begin{equation}\label{negative derivative bound}
 \||\nabla|^{-1}I_NF(u)\|_{L_t^2L_x^{\frac{2d}{d-2}}(J\times\R^d)}
 \lesssim
 \|\nabla I_NF(u)\|_{L_t^2L_x^{\frac{2d}{d+2}}(J\times\R^d)}\lesssim1.
\end{equation}
 Combining
\eqref{energy Holder}--\eqref{negative derivative bound} with
\eqref{local commutator consequence} gives
\[
 |E(w(T))-E(w(t_0))|\lesssim N^{-\delta}
\]
in every dimension $d\ge3$. The same argument applies with $T$
replaced by any $t\in J$, giving the supremum in
\eqref{new energy increment}.
This proves Proposition~\ref{energy increment scattering}.


\begin{thebibliography}{10}
\bibitem{Blowscatter}
J. Bourgain, \emph{Scattering in the energy space and below for 3D NLS}, Journal D'Analyse
Mathematique, \textbf{75} (1998), 267-297.





\bibitem{cwI}
T. Cazenave, F.B. Weissler, \emph{Critical nonlinear Schr\"odinger Equation}, Non. Anal. TMA \textbf{14}
(1990), 807--836.

\bibitem{cazenave:book}
T. Cazenave, \textit{Semilinear Schr\"odinger equations,} Courant
Lecture Notes in Mathematics, \textbf{10}, American Mathematical
Society, 2003.






\bibitem{ckstt:low1}
J. Colliander, M. Keel, G. Staffilani, H. Takaoka, T. Tao, \emph{Almost conservation laws and global
rough solutions to a nonlinear Schr\"odinger equation}, Math. Res. Lett. \textbf{9} (2002), 659--682.







\bibitem{ckstt:low}
J. Colliander, M. Keel, G. Staffilani, H. Takaoka, T. Tao, \emph{Global existence and scattering for rough solutions
of a nonlinear Schr\"odinger equation on $\R^3$}, Comm. Pure Appl. Math. \textbf{57} (2004), 987--1014.

\bibitem{CKSTT-DCDS}
J. Colliander, M. Keel, G. Staffilani, H. Takaoka, T. Tao,
\emph{Resonant decompositions and the $I$-method for the cubic nonlinear Schr\"odinger equation on $\R^2$}. 
Discrete Contin. Dyn. Syst. {\bf 21} (2008), no. 3, 665-685.


\bibitem{Dodson-DCDS}
B. Dodson, 
\emph{Global well-posedness and scattering for the defocusing, cubic nonlinear Schr\"odinger equation when n=3
 via a linear-nonlinear decomposition.}
Discrete Contin. Dyn. Syst. {\bf 33} (2013),  1905-1926.

\bibitem{Dodson-Camb}
B. Dodson, \emph{Global well-posedness and scattering for nonlinear Schrödinger equations with algebraic nonlinearity when d=2,3
 and $u_0$ is radial}. 
Camb. J. Math. {\bf 7} (2019),  283-318.


\bibitem{gv:scatter}
J. Ginibre, G. Velo, \emph{Scattering theory in the energy space for a class of nonlinear Schr\"odinger
equations}, J. Math. Pure. Appl. \textbf{64} (1985), 363--401.


\bibitem{KV-book}
R. Killip and M. Vi\c{s}an \emph{Nonlinear Schrödinger equations at critical regularity}, Evolution equations, Clay Math. Proc., Amer. Math. Soc., Providence, RI, {\bf 17} (2013), 325-437.


\bibitem{keel-tao}
M. Keel, T. Tao, \emph{Endpoint Strichartz Estimates}, Amer. Math. J. \textbf{120} (1998), 955--980.








\bibitem{MMSZheat}
Z.~Ma, C.~Miao, Y.~Song, and J.~Zheng,
\emph{Rough loglog blowup solutions to mass-critical NLS in higher
dimensions $d\ge3$}, preprint (2026).


\bibitem{Su}
Q. Su,
\emph{Global well-posedness and scattering for the defocusing, cubic NLS in $R^3$
.} Math. Res. Lett. {\bf 19} (2012), no. 2, 431-451.

\bibitem{SZ}
Q. Su and Z. Zhao, \emph{Global well-posedness and scattering for the three-dimensional defocusing cubic Schrödinger equation in $H^s$, $s>\frac12$}. arXiv: 2609.24830.



\bibitem{VZ07}
M. Vi\c{s}an and X. Zhang, \emph{On the blowup for the $L^2$-critical focusing nonlinear Schr\"odinger equation in higher dimensions below the energy class}, SIAM J. Math. Anal. {\bf 39} (2007), 34-56. 


\bibitem{VZ09}
M.~Vi\c{s}an and X.~Zhang,
\emph{Global well-posedness and scattering for a class of nonlinear
Schr\"odinger equations below the energy space},  Differ. Integr. Equations, {\bf 22} (2009), 99-124.
\end{thebibliography}
\end{document}